\documentclass{amsart}

\usepackage{amssymb, amsmath, amsthm}
\usepackage{tikz}
\usetikzlibrary{shapes.misc}

\newtheorem{theorem}{Theorem}[section]
\newtheorem{lemma}[theorem]{Lemma}
\newtheorem{proposition}[theorem]{Proposition}
\newtheorem{corollary}[theorem]{Corollary}
\newtheorem{claim}[theorem]{Claim}
\newtheorem{subclaim}[theorem]{Subclaim}
\newtheorem{fact}[theorem]{Fact}

\newenvironment{definition}[1][Definition]{\begin{trivlist}
\item[\hskip \labelsep {\bfseries #1}]}{\end{trivlist}}

\renewcommand{\qedsymbol}{$\dashv$}

\begin{document}
\title{Squares and Covering Matrices}
\author{Chris Lambie-Hanson}
\address{Department of Mathematical Sciences, Carnegie Mellon University \\ Pittsburgh, PA 15213}
\email{clambieh@andrew.cmu.edu}
\date{\today}
\maketitle
\begin{abstract}
 Viale introduced covering matrices in his proof that SCH follows from PFA. In the course of the proof and subsequent work with Sharon, he isolated two reflection principles, CP and S, which, under certain circumstances, are satisfied by all covering matrices of a certain shape. Using square sequences, we construct covering matrices for which CP and S fail. This leads naturally to an investigation of square principles intermediate between $\square_{\kappa}$ and $\square(\kappa^+)$ for a regular cardinal $\kappa$. We provide a detailed picture of the implications between these square principles.
\end{abstract}

\section{Introduction}

There is a fundamental and well-studied tension in set theory between large cardinals and reflection phenomena on the one hand and combinatorial principles (various square principles, in particular) witnessing incompactness (see, for example, \cite{cfm}) on the other. Reflection and large cardinals place limits on the type of combinatorial structures which can exist, and vice versa. 

Covering matrices were introduced by Viale in his proof that the Singular Cardinals Hypothesis follows from the Proper Forcing Axiom \cite{viale}. Here and in later work with Sharon \cite{sharonviale}, Viale also isolated two natural properties, CP($\mathcal{D}$) and S($\mathcal{D}$), which can hold for a given covering matrix $\mathcal{D}$. The statement that CP($\mathcal{D}$) (or S($\mathcal{D}$)) holds for every covering matrix $\mathcal{D}$ of a certain type can be seen as a reflection statement and is thus at odds with the aforementioned incompactness phenomena.

We start this paper by constructing various covering matrices for which CP($\mathcal{D}$) and S($\mathcal{D}$) fail and investigating the relationship between the failure of CP($\mathcal{D}$) and S($\mathcal{D}$) and the existence of square sequences. This leads naturally to the definition of certain square principles which, for a regular, uncountable cardinal $\kappa$, are intermediate between $\square_\kappa$ and $\square(\kappa^+)$. We conclude by obtaining a detailed picture of the implications between these square principles.

Our notation is for the most part standard. Unless otherwise specified, the reference for all notation and definitions is \cite{jech}. If $A$ is a set and $\theta$ is a cardinal, then $[A]^\theta$ is the collection of subsets of $A$ of size $\theta$. If $A$ is a set of ordinals and $\alpha < \sup(A)$ is an ordinal of uncountable cofinality, we say $A$ {\em reflects} at $\alpha$ if $A\cap \alpha$ is stationary in $\alpha$. If $A$ is a set of ordinals, then $A'$ denotes the set of limit ordinals of $A$, i.e. the set of $\alpha$ such that $A\cap \alpha$ is unbounded in $\alpha$, and $\mathrm{otp}(A)$ denotes the order type of $A$. If $\phi: A\rightarrow B$ is a (partial) function and $X\subseteq A$, then $\phi[X]$ is the image of $X$ under $\phi$, and $\phi \restriction X$ is the restriction of $\phi$ to $\mathrm{dom}(\phi)\cap X$. If $\lambda$ is a cardinal and $\mu<\lambda$ is a regular cardinal, then $S^{\lambda}_{\mu} = \{\alpha < \lambda \mid \mathrm{cf}(\alpha)=\mu \}$. $S^{\lambda}_{<\mu}$ is defined in the obvious way. If $s$ is a sequence, then $|s|$ denotes the length of $s$, and, if $t$ is also a sequence, $s^\frown t$ denotes the concatenation of the two.

\section{Covering Matrices}

\begin{definition}
Let $\theta<\lambda$ be regular cardinals. $\mathcal{D}=\{ D(i,\beta) \mid i<\theta, \beta<\lambda \}$ is a {\em $\theta$-covering matrix for $\lambda$} if:
\begin{enumerate}
\item{For all $\beta<\lambda$, $\beta = \bigcup_{i<\theta}D(i,\beta)$.}
\item{For all $\beta<\lambda$ and all $i<j<\theta$, $D(i,\beta) \subseteq D(j, \beta)$.}
\item{For all $\beta<\gamma<\lambda$ and all $i<\theta$, there is $j<\theta$ such that $D(i,\beta)\subseteq D(j,\gamma)$.}
\end{enumerate}

$\beta_{\mathcal{D}}$ is the least $\beta$ such that for all $\gamma<\lambda$ and all $i<\theta$, $\mathrm{otp}(D(i,\gamma))<\beta$. $\mathcal{D}$ is {\em normal} if $\beta_{\mathcal{D}}<\lambda$.

$\mathcal{D}$ is {\em transitive} if, for all $\alpha<\beta<\lambda$ and all $i<\theta$, if $\alpha \in D(i,\beta)$, then $D(i,\alpha)\subseteq D(i,\beta)$.

$\mathcal{D}$ is {\em uniform} if for all $\beta<\lambda$ there is $i<\theta$ such that $D(j,\beta)$ contains a club in $\beta$ for all $j\geq i$. (Note that this is equivalent to the statement that there is $i<\theta$ such that $D(i,\beta)$ contains a club in $\beta$.)

$\mathcal{D}$ is {\em closed} if for all $\beta<\lambda$, all $i<\theta$, and all $X\in [D(i,\beta)]^{\leq \theta}$, $\sup X \in D(i,\beta)$.
\end{definition}

The first part of this paper will be concerned with constructing covering matrices for which the following two reflection properties fail.

\begin{definition}
Let $\theta < \lambda$ be regular cardinals, and let $\mathcal{D}$ be a $\theta$-covering matrix for $\lambda$. 
\begin{enumerate}
\item{$\mathrm{CP}(\mathcal{D}$) holds if there is an unbounded $T\subseteq \lambda$ such that for every $X \in [T]^\theta$, there are $i<\theta$ and $\beta < \lambda$ such that $X\subseteq D(i,\beta)$ (in this case, we say that $\mathcal{D}$ {\em covers} $[T]^\theta$).}
\item{$\mathrm{S}(\mathcal{D}$) holds if there is a stationary $S \subseteq \lambda$ such that for every family $\{S_j \mid j<\theta\}$ of stationary subsets of $S$, there are $i<\theta$ and $\beta < \lambda$ such that, for every $j<\theta$, $S_j \cap D(i, \beta) \not= \emptyset$.}
\end{enumerate}
\end{definition}

\begin{definition}
Let $\theta < \lambda$ be regular cardinals. $\mathrm{R}(\lambda, \theta$) is the statement that there is a stationary $S\subseteq \lambda$ such that for every family $\{S_j \mid j<\theta \}$ of stationary subsets of $S$, there is $\alpha < \lambda$ of uncountable cofinality such that, for all $j<\theta$, $S_j$ reflects at $\alpha$.
\end{definition}

If $\mathcal{D}$ is a nice enough covering matrix, then $\mathrm{CP}(\mathcal{D}$) and $\mathrm{S}(\mathcal{D}$) are equivalent and $\mathrm{R}(\lambda, \theta$) implies both. The following is proved in \cite{sharonviale}:

\begin{lemma}
Let $\theta < \lambda$ be regular cardinals, and let $\mathcal{D}$ be a $\theta$-covering matrix for $\lambda$.
\begin{enumerate}
 \item{If $\mathcal{D}$ is transitive, then $\mathrm{S}(\mathcal{D}$) implies $\mathrm{CP}(\mathcal{D}$).}
 \item{If $\mathcal{D}$ is closed, then $\mathrm{CP}(\mathcal{D}$) implies $\mathrm{S}(\mathcal{D}$).}
 \item{If $\mathcal{D}$ is uniform, then $\mathrm{R}(\lambda, \theta$) implies $\mathrm{S}(\mathcal{D}$).}
\end{enumerate}
\end{lemma}

The following lemma is a key component of Viale's proof that SCH follows from PFA.

\begin{lemma}
 Let $\lambda > \aleph_2$ be a regular cardinal. PFA implies that $\mathrm{CP}(\mathcal{D}$) holds for every $\omega$-covering matrix $\mathcal{D}$ for $\lambda$.
\end{lemma}

Let $\kappa$ be an uncountable regular cardinal. We first consider the question of what types of $\kappa$ covering matrices for $\kappa^+$ can exist. In particular, we will be interested in the existence of a covering matrix that is transitive, normal, and uniform. It turns out that one can always ask for any two of these three properties.

\begin{proposition}
There is a uniform, transitive $\kappa$-covering matrix for $\kappa^+$.
\end{proposition}

\begin{proof}
Simply let $D(i,\beta)=\beta$ for every $i<\kappa$ and $\beta<\kappa^+$.
\end{proof}

\begin{proposition}
There is a transitive, normal $\kappa$-covering matrix for $\kappa^+$.
\end{proposition}

\begin{proof}
For each $\alpha<\kappa^+$, let $\phi_\alpha : \kappa \rightarrow \alpha$ be a surjection. For $i<\kappa$ and $\beta<\kappa^+$, recursively define \[D(i,\beta)=\phi_\beta [i] \cup \bigcup_{\alpha \in \phi_\beta[i]}D(i, \alpha).\] Let $\mathcal{D}=\{D(i,\beta) \mid i<\kappa, \beta<\kappa^+\}$.

It is clear that $\mathcal{D}$ is a covering matrix and, inductively, $|D(i, \beta) | < \kappa$ for every $i$ and $\beta$. Thus, $\beta_{\mathcal{D}} \leq \kappa$, so $\mathcal{D}$ is normal. For each $i<\kappa$, we show by induction on $\beta <\kappa^+$ that if $\alpha \in D(i,\beta)$, then $D(i,\alpha)\subseteq D(i,\beta)$. Indeed, if $\alpha \in \phi_\beta [i]$, then the conclusion holds by definition, while if $\alpha \in D(i,\gamma)$ for some $\gamma \in \phi_\beta [i]$, then by induction $D(i,\alpha) \subseteq D(i,\gamma) \subseteq D(i,\beta)$. Thus, $\mathcal{D}$ is transitive. 
\end{proof}

The following lemma on ordinal arithmetic will be quite useful in our construction of covering matrices.

\begin{lemma}
Let $\theta$ be a regular cardinal, $\mu<\theta$, and $m<\omega$. Suppose that for each $i<\mu$, $X_i$ is a set of ordinals such that $\mathrm{otp}(X_i)<\theta^m$. Let $X=\bigcup_{i<\mu}X_i$. Then $\mathrm{otp}(X)<\theta^m$.
\end{lemma}

\begin{proof}
By induction on $m$. The conclusion is immediate for $m=0$ and $m=1$. Let $m\geq 2$ and suppose for sake of contradiction that $\mathrm{otp}(X)\geq \theta^m$. Fix $A\subseteq X$ of order type exactly $\theta^m$. Enumerate $A$ in increasing order as $A=\{a_\alpha \mid \alpha<\theta^m\}$. For each $\beta<\theta$, let $A_\beta = \{a_{\theta^{m-1} \cdot \beta +\gamma} \mid \gamma <\theta^{m-1}\}$. Then $\mathrm{otp}(A_\beta)=\theta^{m-1}$, so by the induction hypothesis, there is $i_\beta<\mu$ such that $\mathrm{otp}(X_{i_\beta}\cap A_\beta)=\theta^{m-1}$. Thus, there is an $i^*<\mu$ such that $i_\beta=i^*$ for unboundedly many $\beta<\theta$. But then $\mathrm{otp}(X_{i^*})\geq \theta^m$. Contradiction. 
\end{proof}

\begin{proposition}
There is a uniform, normal $\kappa$-covering matrix for $\kappa^+$.
\end{proposition}

\begin{proof}
For each $\alpha<\kappa^+$, let $C_\alpha$ be a club in $\alpha$ such that $\mathrm{otp}(C_\alpha)\leq \kappa$, and let $\phi_\alpha :\kappa \rightarrow \alpha$ be a surjection. We define $\mathcal{D}=\{D(i,\beta) \mid i<\kappa, \beta<\kappa^+\}$ by recursion on $\beta$ and, for fixed $\beta$, by recursion on $i$. For each $\beta<\kappa^+$, let $D(0,\beta)=C_\beta$. If $i<\kappa$ is a limit ordinal, let \[D(i,\beta)=\bigcup_{j<i}D(j,\beta).\] Finally, let \[D(i+1,\beta)=D(i,\beta)\cup \phi_\beta[i]\cup \bigcup_{\alpha \in \phi_\beta[i]} D(i+1,\alpha).\]

It is easily verified that $\mathcal{D}$ is a $\kappa$-covering matrix for $\kappa^+$ and, by construction, $D(0,\beta)$ contains a club in $\beta$ for each $\beta<\kappa^+$. It remains to show that $\mathcal{D}$ is normal. We in fact prove by induction on $\beta<\kappa^+$ and, for fixed $\beta$, by induction on $i<\kappa$, that $\mathrm{otp}(D(i,\beta))<\kappa^2$ for all $i$ and $\beta$. Fix $i<\kappa$ and $\beta<\kappa^+$. By the inductive hypothesis, $D(i,\beta)$ is a union of fewer than $\kappa$-many sets, all of which have order type less than $\kappa^2$. Then, by the previous lemma, $\mathrm{otp}(D(i,\beta))<\kappa^2$. Thus, $\beta_{\mathcal{D}}\leq \kappa^2 <\kappa^+$, so $\mathcal{D}$ is normal. 
\end{proof}

However, we can not always get all three properties, since $\mathrm{CP}(\mathcal{D}$) and $\mathrm{S}(\mathcal{D}$) necessarily fail for a transitive, normal, uniform $\kappa$-covering matrix for $\kappa^+$.

\begin{lemma}
If $\mathcal{D}$ is a normal $\kappa$-covering matrix for $\kappa^+$, then $\mathrm{CP}(\mathcal{D}$) fails.
\end{lemma}

\begin{proof}
Let $T$ be an unbounded subset of $\kappa^+$. Then any $X \in [T]^{\kappa}$ whose order type is greater than $\beta_{\mathcal{D}}$ can not be contained in any $D(i,\beta)$. 
\end{proof}
\ 
Since $\mathrm{S}(\mathcal{D}$) implies $\mathrm{CP}(\mathcal{D})$ whenever $\mathcal{D}$ is transitive, $\mathrm{S}(\mathcal{D}$) fails for every transitive, normal $\kappa$-covering matrix for $\kappa^+$.

\begin{proposition}
If $\mathrm{R}(\kappa^+, \kappa)$ holds, then there are no transitive, normal, uniform $\kappa$-covering matrices for $\kappa^+$.
\end{proposition}

\begin{proof}
$\mathrm{R}(\kappa^+, \kappa)$ implies that every uniform $\kappa$-covering matrix $\mathcal{D}$ for $\kappa^+$ satisfies $\mathrm{S}(\mathcal{D}$).  But we saw above that $\mathrm{S}(\mathcal{D}$) fails for every transitive, normal $\kappa$-covering matrix for $\kappa^+$. 
\end{proof}

\begin{corollary}
If MM holds, then there are no transitive, normal, uniform $\omega_1$-covering matrices for $\omega_2$.
\end{corollary}

\begin{proof}
 MM implies that $\mathrm{R}(\aleph_n, \aleph_1$) holds for every $1<n<\omega$ as witnessed by $S^{\aleph_n}_{\aleph_0}$.
\end{proof}

The existence of a transitive, normal, uniform covering matrix does follow, however, from sufficiently strong square principles.

\begin{proposition}
Suppose $\kappa$ is a regular cardinal and $\square_{\kappa, <\kappa}$ holds. Then there is a transitive, normal, uniform $\kappa$-covering matrix for $\kappa^+$.
\end{proposition}

\begin{proof}
Let $\langle \mathcal{C}_\alpha \mid \alpha \in \mathrm{lim}(\kappa^+)\rangle$ be a $\square_{\kappa, <\kappa}$-sequence. We construct a transitive, normal, uniform $\kappa$-covering matrix for $\kappa^+$, $\mathcal{D}=\{D(i,\beta):i<\kappa,\  \beta<\kappa^+\}$, by recursion on $\beta$ as follows:
\begin{itemize}
\item{$D(i,\beta+1)=D(i,\beta)\cup \{\beta\}$}
\item{If $\beta$ is a limit ordinal and $\mathrm{cf}(\beta)<\kappa$, fix $E$, a club in $\beta$ of order type less than $\kappa$, and let \[D(i,\beta) = 
\begin{cases}
\emptyset  & \text{if } \sup_{C\in \mathcal{C}_\beta}\mathrm{otp}(C)\geq \omega \cdot i \\
E \cup \bigcup_{\alpha \in E} D(i,\alpha) & \text{if } \sup_{C\in \mathcal{C}_\beta}\mathrm{otp}(C)< \omega \cdot i
\end{cases}\]}
\item{If $\mathrm{cf}(\beta)=\kappa$, fix $C\in \mathcal{C}_\beta$, and let \[D(i,\beta)=C' \cup \bigcup_{\alpha \in C'}D(i,\alpha)\]}
\end{itemize}
It is routine to check that $\mathcal{D}$ is a uniform $\kappa$-covering matrix for $\kappa^+$, and an easy induction shows that it is transitive. We claim that $\mathcal{D}$ is normal. We prove that $\mathrm{otp}(D(i,\beta))<\kappa^2$ for every $i<\kappa$ and $\beta<\kappa^+$ by induction on $\beta$. If $\beta$ is a successor ordinal or a limit ordinal of cofinality less than $\kappa$, then $D(i,\beta)$ is the union of fewer than $\kappa$-many sets, each, by the induction hypothesis, of order type less than $\kappa^2$. Thus, $\mathrm{otp}(D(i,\beta))<\kappa^2$. Suppose $\mathrm{cf}(\beta)=\kappa$. Let $C\in \mathcal{C_\beta}$ be the club used in the construction of $D(i,\beta)$. Enumerate $C'$ in increasing order as $\{\alpha_\gamma \mid \gamma<\kappa\}$. For each $\gamma<\kappa$, $C \cap \alpha_\gamma \in \mathcal{C}_{\alpha_\gamma}$, so for $\gamma \geq i$, $D(i, \alpha_\gamma) = \emptyset$. Thus, $D(i,\beta)$ is itself a union of fewer than $\kappa$-many sets, each of order type less than $\kappa^2$, so $\mathrm{otp}(D(i,\beta))<\kappa^2$. 
\end{proof}

We now show that $\square_{\kappa, <\kappa}$ is the optimal hypothesis in the previous proposition by producing, via a standard argument due originally to Baumgartner \cite{baumgartner}, a model in which $\square^*_\kappa$ and $\mathrm{R}(\kappa^+, \kappa)$ both hold. We need the following lemma, due to Shelah.

\begin{lemma}
\label{aplem}
Let $\mu \leq \kappa$ be regular cardinals, let $S\subseteq S^{\kappa^+}_{<\mu}$ be stationary, and let $\mathbb{P}$ be a $\mu$-closed forcing poset. If $G$ is $\mathbb{P}$-generic over $V$, then $S$ is stationary in $V[G]$.
\end{lemma}

\begin{proposition}
Let $\kappa<\lambda$, with $\kappa$ regular and $\lambda$ measurable. Let $\mathbb{P}=\mathrm{Coll}(\kappa, <\lambda)$, and let $G$ be $\mathbb{P}$-generic over $V$. Then, in $V[G]$, $\kappa^{<\kappa}=\kappa$ and $R(\lambda, \kappa)$ hold.
\end{proposition}

\begin{proof}
First note that, since $\mathbb{P}$ is $\kappa$-closed, it doesn't add any new bounded subsets of $\kappa$, so, since $\lambda$ is measurable in $V$, $\kappa^{<\kappa}=\kappa$ in $V[G]$, so $\square^*_\kappa$ holds in $V[G]$. 

 We now show that $R(\lambda, \kappa)$ holds in $V[G]$ as witnessed by $S^\lambda_{<\kappa}$. Let $j:V\rightarrow M$ be an elementary embedding with $M$ transitive and $\mathrm{crit}(j)=\lambda$, and let $H$ be $\mathrm{Coll}(\kappa, <j(\lambda))$-generic over $V$ such that $G\subset H$. We can then lift the embedding to $j:V[G]\rightarrow M[H]$. Let $\{S_\alpha \mid \alpha<\kappa\}$ be a family of stationary subsets of $S^\lambda_{<\kappa}$. Note that $j(\{S_\alpha \mid \alpha<\kappa\})=\{j(S_\alpha) \mid \alpha<\kappa\}$ and, for each $\alpha<\kappa$, $j(S_\alpha)\cap \lambda = S_\alpha$. Since $\mathrm{Coll}(\kappa, [\lambda, j(\lambda))$ is $\kappa$-closed, each $S_\alpha$ remains stationary in $V[H]$ and therefore also in $M[H]$. Thus, in $M[H]$, the sets $\{j(S_\alpha) \mid \alpha<\kappa\}$ reflect simultaneously to a point of uncountable cofinality below $j(\lambda)$, namely $\lambda$, so, by elementarity, in $V[G]$, the sets $\{S_\alpha \mid \alpha<\kappa \}$ reflect simultaneously to a point of uncountable cofinality below $\lambda$. Thus, in $V[G]$, $R(\lambda, \kappa)=R(\kappa^+, \kappa)$ holds. 
\end{proof}

Thus, $\square^*_\kappa$ does not imply the existence of a transitive, normal, uniform $\kappa$-covering matrix for $\kappa^+$. 

We now prove that the converse of Proposition 2.10 does not hold in general by showing that, if $\kappa$ is regular and not strongly inaccessible, one can force to add a transitive, normal, uniform $\kappa$-covering matrix $\mathcal{D}$ for $\kappa^+$ without adding a $\square_{\kappa, <\kappa}$-sequence. The argument is similar to that introduced by Jensen to distinguish between various weak square principles (see \cite{jensen}). Let $\mathbb{Q}$ be the forcing poset consisting of conditions of the form $q=\{D^q(i,\beta) \mid i<\kappa, \  \beta \leq \beta^q\}$ such that:
\begin{itemize}
\item{$\beta^q<\kappa^+$.}
\item{For all $\beta \leq \beta^q$, $\beta = \bigcup_{i<\kappa}D^q(i,\beta)$.}
\item{For all $\beta \leq \beta^q$ and all $i<j<\kappa$, $D^q(i,\beta)\subseteq D^q(j,\beta)$.}
\item{For all $\alpha<\beta \leq \beta^q$ and all $i<\kappa$, if $\alpha \in D^q(i,\beta)$, then $D^q(i,\alpha)\subseteq D^q(i,\beta)$.}
\item{For all $i<\kappa$ and all $\beta \leq \beta^q$, $\mathrm{otp}(D^q(i,\beta))<\kappa^2$.}
\item{For all $\beta \leq \beta^q$, $D^q(i,\beta)$ contains a club in $\beta$ for sufficiently large $i<\kappa$.}
\end{itemize}
For $p,q\in \mathbb{Q}$, $p\leq q$ if and only if $p$ end-extends $q$, i.e. $\beta^p \geq \beta^q$ and $D^p(i,\beta)=D^q(i,\beta)$ for every $i<\kappa$ and $\beta \leq \beta^q$.

\begin{proposition}
$\mathbb{Q}$ is $\kappa$-closed.
\end{proposition}

\begin{proof}
Suppose $\mu < \kappa$ and $\langle q_\alpha \mid \alpha < \mu \rangle$ is a descending sequence of conditions from $\mathbb{Q}$. We will define $q\in \mathbb{Q}$ such that for all $\alpha < \mu$, $q\leq q_\alpha$. Let $\beta^q=\sup(\{\beta^{q_\alpha} \mid \alpha <\mu \})$. We may assume without loss of generality that the $\beta^{q_\alpha}$ were strictly increasing, so, for all $\alpha < \mu$, $\beta^{q_\alpha}<\beta^q$. For all $\beta < \beta^q$ and $i<\kappa$, let $D^q(i,\beta)=D^{q_\alpha}(i,\beta)$ for some $\alpha < \mu$ such that $\beta \leq \beta^{q_{\alpha}}$. It remains to define $D^q(i,\beta^q)$ for $i<\kappa$. To this end, fix a club $C\subseteq \beta^q$ whose order type is cf($\beta^q$). Note that $|C| < \kappa$. For $i<\kappa$, let $D^q(i,\beta^q)=C\cup \bigcup_{\beta \in C}D^q(i,\beta)$. Since $D^q(i,\beta^q)$ is the union of fewer than $\kappa$-many sets of order type less than $\kappa^2$, the order type of $D^q(i,\beta^q)$ is also less than $\kappa^2$. It easily follows that $q\in \mathbb{Q}$ and, for all $\alpha < \mu$, $q\leq q_\alpha$. 
\end{proof}

We now need the notion of {\em strategic closure}.

\begin{definition}
Let $\mathbb{P}$ be a partial order and let $\beta$ be an ordinal.
\begin{enumerate}
\item {The two-player game $G_\beta(\mathbb{P})$ is defined as follows: Players I and II alternately play entries in $\langle p_\alpha \mid \alpha < \beta \rangle$, a decreasing sequence of conditions in $\mathbb{P}$ with $p_0 = 1_{\mathbb{P}}$. Player I plays at odd stages, and Player II plays at even stages (including all limit stages). If there is an even stage $\alpha < \beta$ at which Player II can not play, then Player I wins. Otherwise, Player II wins.}
\item{$\mathbb{P}$ is said to be {\em $\beta$-strategically closed} if Player II has a winning strategy for the game $G_\beta(\mathbb{P})$.}
\end{enumerate}
\end{definition}

The following is immediate.

\begin{fact}
Let $\mathbb{P}$ be a partial order and let $\kappa$ be a cardinal. If $\mathbb{P}$ is $(\kappa +1)$-strategically closed, then forcing with $\mathbb{P}$ does not add any new $\kappa$-sequences of ordinals.
\end{fact}

\begin{proposition}
$\mathbb{Q}$ is $(\kappa +1)$-strategically closed.
\end{proposition}

\begin{proof}
We need to exhibit a winning strategy for Player II in the game $G_{\kappa+1}(\mathbb{Q})$. Suppose $\gamma \leq \kappa$ is an even or limit ordinal and that $\langle q_\alpha \mid \alpha<\gamma \rangle$ has been played. We specify Player II's next move, $q_\gamma$. Let $C_\gamma = \{\beta^{q_\alpha} \mid \alpha <\gamma \mbox{ is an even or limit}$ $\mbox{ordinal} \}$ ($C_\gamma$ is thus the set of the top points of the conditions played by Player II thus far). We assume the following induction hypotheses are satisfied:
\begin{enumerate}
\item{$C_\gamma$ is closed beneath its supremum.}
\item{If $\alpha<\alpha'<\gamma$ and $\alpha,\alpha'$ are even ordinals, then $\beta^{q_\alpha}<\beta^{q_{\alpha'}}$.}
\item{For all even ordinals $\alpha <\gamma$ and all $i<\alpha$, $D^{q_\alpha}(i,\beta^{q_\alpha})=\emptyset$.}
\end{enumerate}
There are three cases.

{\bf Case 1: $\gamma$ is a successor ordinal:} Suppose $\gamma = \gamma'+1$. Let $\beta^{q_\gamma}=\beta^{q_{\gamma'}}+1$. For $i<\kappa$ and $\beta \leq \beta^{q_{\gamma'}}$, let $D^{q_\gamma}(i,\beta)=D^{q_{\gamma'}}(i,\beta)$. For $i<\kappa$, let \[D^{q_\gamma}(i,\beta^{q_\gamma})=
\begin{cases}
\emptyset  & \text{if } i<\gamma \\
\{\beta^{q_{\gamma'}}\}\cup D^{q_{\gamma'}}(i,\beta^{q_{\gamma'}}) & \text{if } i \geq \gamma
\end{cases}.\]

{\bf Case 2: $\gamma<\kappa$ is a limit ordinal:} Let $\beta^{q_\gamma}=\sup(C_\gamma)$ (so $C_\gamma$ is club in $\beta^{q_\gamma}$). For $i<\kappa$ and $\beta < \beta^{q_\gamma}$, let $D^{q_\gamma}(i,\beta)=D^{q_{\alpha}}(i,\beta)$ for some $\alpha <\gamma$ such that $\beta \leq \beta^{q_\alpha}$. For $i<\kappa$, let \[D^{q_\gamma}(i,\beta^{q_\gamma})=
\begin{cases}
\emptyset  & \text{if } i<\gamma \\
C_\gamma \cup \bigcup_{\beta \in C_\gamma}D^{q_\gamma}(i,\beta) & \text{if } i \geq \gamma
\end{cases}.\] For all $i<\kappa$, $D^{q_\gamma}(i,\beta^{q_\gamma})$ is the union of fewer than $\kappa$-many sets of order type less than $\kappa^2$ and thus has order type less than $\kappa^2$.

{\bf Case 3: $\gamma = \kappa$:} Let $\beta^{q_\gamma}=\sup(C_\gamma)$. For $i<\kappa$ and $\beta < \beta^{q_\gamma}$, let $D^{q_\gamma}(i,\beta)=D^{q_{\alpha}}(i,\beta)$ for some $\alpha <\gamma$ such that $\beta \leq \beta^{q_\alpha}$. For $i<\kappa$, let \[D^{q_\gamma}(i,\beta^{q_\gamma})=C_\gamma \cup \bigcup_{\beta \in C_\gamma}D^{q_\gamma}(i,\beta).\] Since, for each $i<\kappa$, $D^{q_\gamma}(i,\beta)=\emptyset$ for all $\beta \in C_\gamma \setminus \beta^{q_i}$, each $D^{q_\gamma}(i,\beta^{q_\gamma})$ is the union of fewer than $\kappa$-many sets of order type less than $\kappa^2$ and thus has order type less than $\kappa^2$.

It is easy to check that in each case the inductive hypotheses are preserved and that this provides a winning strategy for Player II in $G_{\kappa +1}(\mathbb{Q})$. Thus, $\mathbb{Q}$ is $(\kappa +1)$-strategically closed. 
\end{proof}

\begin{proposition}
If $2^\kappa = \kappa^+$, then $\mathbb{Q}$ is a cardinal-preserving forcing poset that adds a transitive, normal, uniform $\kappa$-covering matrix for $\kappa^+$.
\end{proposition}

\begin{proof}
Since $\mathbb{Q}$ is $(\kappa + 1)$-strategically closed, forcing with $\mathbb{Q}$ does not add any new $\kappa$-sequences of ordinals, so all cardinals $\leq \kappa^+$ are preserved. Since $2^\kappa = \kappa^+$, $|\mathbb{Q}| = \kappa^+$, so $\mathbb{Q}$ has the $\kappa^{++}$-chain condition and hence preserves all cardinals $\geq \kappa^{++}$. Finally, a proof similar to that of the previous proposition yields the fact that for all $\alpha < \kappa^+$, the set $E_\alpha = \{q \mid \beta^q \geq \alpha \}$ is dense in $\mathbb{Q}$. Thus, if $G$ is $\mathbb{Q}$-generic over $V$, then $\bigcup G$ is a transitive, normal, uniform $\kappa$-covering matrix for $\kappa^+$. 
\end{proof}

Given a $\kappa$-covering matrix $\mathcal{D}$ for $\kappa^+$, we define a forcing notion $\mathbb{T}_\mathcal{D}$ whose purpose is to add a club in $\kappa^+$ of order type $\kappa$ which interacts nicely with $\mathcal{D}$. It plays a similar role in our argument as the forcing to thread a square sequence plays in \cite{cfm} and \cite{jensen}. Elements of $\mathbb{T}_\mathcal{D}$ are sets $t$ such that:
\begin{enumerate}
\item{$t$ is a closed, bounded subset of $\kappa^+$.}
\item{$|t|<\kappa$.}
\item{If $t$ is enumerated in increasing order as $\langle \tau_\alpha \mid \alpha \leq \gamma_t <\kappa \rangle$, then for all $\alpha \leq \gamma_t$ and all $i<\alpha$, $D(i,\tau_\alpha)=\emptyset$.}
\end{enumerate}

If $t,t' \in \mathbb{T}_\mathcal{D}$, then $t'\leq t$ if and only if $t'$ end-extends $t$, i.e. $\gamma_{t'} \geq \gamma_t$ and, for all $\alpha \leq \gamma_t$, $\tau'_\alpha = \tau_\alpha$.

In general, $\mathbb{T}_\mathcal{D}$ may be very poorly behaved. For example, the set it adds may not be cofinal in $\kappa^+$ and, even if it is, its order type might be less than $\kappa$. However, if $\mathcal{D}$ has been added by $\mathbb{Q}$ immediately prior to forcing with $\mathbb{T}_{\mathcal{D}}$, then it has some nice properties.

\begin{proposition}
Let $G$ be $\mathbb{Q}$-generic over $V$, and let $\mathcal{D}=\bigcup G$. Then, in $V[G]$, for all $\alpha<\kappa^+$, $E_\alpha = \{t \mid \alpha \leq \tau_{\gamma_t} \}$ is dense in $\mathbb{T}_\mathcal{D}$. 
\end{proposition}

\begin{proof}
This follows from the fact that, in $V$, for every $\alpha < \kappa^+$ and every $j<\kappa$, the set $E_{j, \alpha} = \{q \mid \alpha \leq \beta^q \mbox{ and for every } i<j, D(i,\beta^q)=\emptyset \}$ is easily seen to be dense in $\mathbb{Q}$. 
\end{proof}

\begin{proposition}
If $\mathcal{D}$ is the covering matrix added by $\mathbb{Q}$, then $\mathbb{Q}*\dot{\mathbb{T}}_\mathcal{D}$ has a $\kappa$-closed dense subset.
\end{proposition}

\begin{proof}
Let $\mathbb{S} = \{(q,\dot{t}) \mid q \mbox{ decides the value of } \dot{t} \mbox{ and } q \Vdash ``\dot{\tau}_{\gamma_t}=\beta^q" \}$. We first show that $\mathbb{S}$ is dense in $\mathbb{Q}*\dot{\mathbb{T}}_\mathcal{D}$. To this end, let $(q_0, \dot{t}_0) \in \mathbb{Q}*\dot{\mathbb{T}}_\mathcal{D}$. Find $q_1 \leq q_0$ such that $q_1$ decides the value of $\dot{t}$ to be some $\langle \tau_\alpha \mid \alpha \leq \gamma_t <\kappa \rangle$ (this is possible, since $\mathbb{Q}$ is $(\kappa + 1)$-strategically closed and hence doesn't add any new $\kappa$-sequences of ordinals). Without loss of generality, $\beta^{q_1} > \tau_{\gamma_t}$. Now form $q_2 \leq q_1$ by setting $\beta^{q_2}=\beta^{q_1}+1$ and \[D^{q_2}(i,\beta^{q_2}) = 
\begin{cases}
\emptyset  & \text{if } i \leq \gamma_t \\
\{\beta^{q_1}\} \cup D^{q_1}(i,\beta^{q_1}) & \text{if } i > \gamma_t
\end{cases}.\] Finally, let $\dot{t}_1$ be such that $q_2 \Vdash \dot{t}_1=\dot{t}_0 \cup \{\beta^{q_2} \}$. Then $(q_2, \dot{t}_1)\leq (q_0, \dot{t}_0)$ and $(q_2, \dot{t}_1) \in \mathbb{S}$.

Next, we show that $\mathbb{S}$ is $\kappa$-closed. Let $\langle (q_\alpha, \dot{t}_\alpha) \mid \alpha<\nu \rangle$ be a decreasing sequence of conditions from $\mathbb{S}$ with $\nu<\kappa$ a limit ordinal. We will find a lower bound $(q, \dot{t}) \in \mathbb{S}$. Let $\beta^q = \sup(\{\beta^{q_\alpha} \mid \alpha<\nu \})$ and let $X=\{\beta \mid \mbox{for some } \alpha<\nu, q_\alpha \Vdash ``\beta \in \dot{t}_\alpha" \}$. Note that by our definition of $\mathbb{S}$, $X$ is club in $\beta^q$. Let $\gamma = \mathrm{otp}(X)$. Define $q$ as a lower bound to the $q_\alpha$'s by letting \[D^q(i,\beta^q) = 
\begin{cases}
\emptyset  & \text{if } i \leq \gamma \\
X \cup \bigcup_{\beta \in X}D^q(i,\beta) & \text{if } i > \gamma
\end{cases}.\] Let $\dot{t}$ be a name forced by $q$ to be equal to $X\cup \{\beta^q \}$. Then $(q,\dot{t}) \in \mathbb{S}$ and, for all $\alpha <\nu$, $(q,\dot{t})\leq (q_\alpha, \dot{t}_\alpha)$. 
\end{proof}

Thus, if $\mathcal{D}$ has been added by $\mathbb{Q}$, then $\mathbb{T}_\mathcal{D}$ does in fact add a club in $\kappa^+$ and, since $\mathbb{Q}*\dot{\mathbb{T}}_\mathcal{D}$ has a $\kappa$-closed dense subset and therefore doesn't add any new sets of ordinals of order type less than $\kappa$, the club added by $\mathbb{T}_\mathcal{D}$ has order type $\kappa$.

We will need the following fact, due to Magidor (see \cite{magidor}):

\begin{fact}
\label{lift}
Let $\kappa$ be a regular cardinal, and let $\kappa<\lambda<\mu$. Suppose that, in $V^{\mathrm{Coll}(\kappa, <\lambda)}$, $\mathbb{P}$ is a $\kappa$-closed partial order and $|\mathbb{P}|<\mu$. Let $i$ be the natural complete embedding of $\mathrm{Coll}(\kappa, <\lambda)$ into $\mathrm{Coll}(\kappa, <\mu)$ (namely, the identity embedding). Then $i$ can be extended to a complete embedding $j$ of $\mathrm{Coll}(\kappa, <\lambda)*\mathbb{P}$ into $\mathrm{Coll}(\kappa, <\mu)$ so that the quotient forcing $\mathrm{Coll}(\kappa, <\mu)/j[\mathrm{Coll}(\kappa, <\lambda)*\mathbb{P}]$ is $\kappa$-closed.
\end{fact}

\begin{theorem}
Let $\kappa$ be a regular cardinal that is not strongly inaccessible, and let $\lambda > \kappa$ be a measurable cardinal. Let $G$ be $\mathrm{Coll}(\kappa, <\lambda)$-generic over $V$ and, in $V[G]$, let $H$ be $\mathbb{Q}$-generic over $V[G]$. Then, in $V[G*H]$, there is a transitive, normal, uniform $\kappa$-covering matrix for $\kappa^+$, but $\square_{\kappa, <\kappa}$ fails.
\end{theorem}

\begin{proof}
We have already shown that, in $V[G*H]$, there is a transitive, normal, uniform $\kappa$-covering matrix for $\kappa^+$, $\mathcal{D}$. In $V[G*H]$, let $\mathbb{T}=\mathbb{T}_\mathcal{D}$. Fix an elementary embedding $j:V\rightarrow M$ with critical point $\lambda$. Then $j(\mathbb{P})=\mathrm{Coll}(\kappa, <j(\lambda))$, and $j \restriction \mathbb{P}$ is the identity map. $V[G] \models |\mathbb{Q}*\mathbb{T}|=\lambda$, and $|\mathbb{Q}*\mathbb{T}|$ has a $\kappa$-closed dense subset, so we can extend $j \restriction \mathbb{P}$ to a complete embedding of $\mathrm{Coll}(\kappa, <\lambda)*\mathbb{Q}*\mathbb{T}$ into $\mathrm{Coll}(\kappa, <j(\lambda))$ so that the quotient forcing is $\kappa$-closed. Then, letting $I$ be $\mathbb{T}$-generic over $V[G*H]$ and $J$ be $\mathbb{R}=\mathrm{Coll}(\kappa, <j(\lambda))/G*H*I$-generic over $V[G*H*I]$, we can further extend $j$ to an elementary embedding $j:V[G] \rightarrow M[G*H*I*J]$.

We would now like to extend $j$ further still to an embedding with domain $V[G*H]$. To do this, consider the partial order $j(\mathbb{Q})$. In $M[G*H*I*J]$, $j(\mathbb{Q})$ is the partial order to add a transitive, normal, uniform $\kappa$-covering matrix for $j(\lambda)$. Let $E=\bigcup I$. $E$ is a club in $\lambda$ of order type $\kappa$, and if $\beta \in E$ is such that $\mathrm{otp}(E\cap \beta) = \gamma$, then for every $i<\gamma$, $D(i,\beta)=\emptyset$.  We use $E$ to define a ``master condition" $q^* \in j(\mathbb{Q})$ as follows.  Let $\beta^{q^*}=\lambda$. For $\beta<\lambda$ and $i<\kappa$, let $D^{q^*}(i, \beta) = D(i,\beta)$ and \[D^{q*}(i,\lambda)=E\cup \bigcup_{\beta \in E} D(i,\beta).\] $q^* \in j(\mathbb{Q})$ and, if $q\in H$, $j(q)=q\leq q^*$. Let $K$ be $j(\mathbb{Q})$-generic over $V[G*H*I*J]$ such that $q^* \in K$. Since $j[H] \subseteq K$, we can extend $j$ to an elementary embedding $j:V[G*H] \rightarrow M[G*H*I*J*K]$.

Suppose for sake of contradiction that $\overrightarrow{\mathcal{C}} = \langle \mathcal{C}_\alpha \mid \alpha < \lambda \rangle$ is a $\square_{\kappa, <\kappa}$-sequence in $V[G*H]$. For $\alpha<\lambda$, $j(\mathcal{C}_\alpha) = \mathcal{C}_\alpha$, and $j(\overrightarrow{\mathcal{C}})=\langle \mathcal{C}_\alpha \mid \alpha <j(\lambda) \rangle$ is a $\square_{\kappa, <\kappa}$-sequence in $M[G*H*I*J*K]$. Fix $F\in \mathcal{C}_\lambda$. $F$ is a thread through $\overrightarrow{\mathcal{C}}$ (i.e., $F$ is a club in $\lambda$ and, for every $\alpha \in F'$, $F\cap \alpha \in \mathcal{C}_\alpha$) and $F \in V[G*H*I*J*K]$.
\begin{claim}
$F\in V[G*H]$.
\end{claim}

Suppose not. Work in $V[G]$. There is a $\mathbb{Q}*\mathbb{T}*\mathbb{R}*j(\mathbb{Q})$-name $\dot{f}$ such that $\dot{f}^{H*I*J*K}=F$ and $\Vdash_{\mathbb{Q}*\mathbb{T}*\mathbb{R}*j(\mathbb{Q})} \dot{f} \not\in V[G*G_\mathbb{Q}]$.

\begin{subclaim}
For all $(q, \dot{t}, \dot{r}, \dot{p}) \in \mathbb{Q}*\mathbb{T}*\mathbb{R}*j(\mathbb{Q})$, there are $q' \leq q$, $(\dot{t}_0, \dot{r}_0, \dot{p}_0)$, $(\dot{t}_1, \dot{r}_1, \dot{p}_0)$, and $\alpha <\kappa^+$ such that $(q', \dot{t}_0, \dot{r}_0, \dot{p}_0), (q', \dot{t}_1, \dot{r}_1, \dot{p}_1) \leq (q,\dot{t}, \dot{r}, \dot{p})$ and such that $(q', \dot{t}_0, \dot{r}_0, \dot{p}_0)$ and $(q', \dot{t}_1, \dot{r}_1, \dot{p}_1)$ decide the statement $``\check{\alpha} \in \dot{f}"$ in opposite ways.
\end{subclaim}

Suppose the subclaim fails for some $(q,\dot{t}, \dot{r}, \dot{p})$. Define a $\mathbb{Q}$-name $\dot{f}'$ such that for all $q' \leq q$ and $\alpha < \kappa^+$, $q' \Vdash_{\mathbb{Q}} ``\check{\alpha} \in \dot{f}"$ if and only if there is $(\dot{t}', \dot{r}', \dot{p}')$ such that $(q', \dot{t}', \dot{r}', \dot{p}') \leq (q', \dot{t}, \dot{r}, \dot{p})$ and $(q',\dot{t}', \dot{r}', \dot{p}') \Vdash_{\mathbb{Q}*\mathbb{T}*\mathbb{R}*j(\mathbb{Q})} ``\check{\alpha} \in \dot{f}"$. Then $(q,\dot{t}, \dot{r}, \dot{p}) \Vdash ``\dot{f} = \dot{f}'"$, contradicting our choice of $\dot{f}$. This proves the subclaim.

Since $\kappa$ is not strongly inaccessible, letting $\gamma$ be the least cardinal such that $2^\gamma \geq \kappa$, we have $\gamma < \kappa$. Recall that $\mathbb{S}$ is the previously defined $\kappa$-closed dense subset of $\mathbb{Q}*\mathbb{T}$. We will construct $\langle q_i \mid i \leq \gamma \rangle$, $\langle (\dot{t}_s, \dot{r}_s, \dot{p}_s) \mid s \in {^{\leq \gamma} 2} \rangle$ and $\langle \alpha_i \mid i \leq \gamma \rangle$ such that:
\begin{enumerate}
\item{$(q_0, \dot{t}_{\langle \rangle}, \dot{r}_{\langle \rangle}, \dot{p}_{\langle \rangle}) \Vdash ``\dot{f}$ is a thread through $\mathcal{C}$".}
\item{For each $s \in {^{\leq \gamma} 2}$, $(q_{|s|}, \dot{t}_s, \dot{r}_s, \dot{p}_s) \in \mathbb{S}*\mathbb{R}*j(\mathbb{Q})$ and $\beta^{q_{|s|}}=\alpha_{|s|}$.}
\item{If $s,u \in {^{\leq \gamma} 2}$ and $s \subseteq u$, then $(q_{|u|}, \dot{t}_u, \dot{r}_u, \dot{p}_u) \leq (q_{|s|}, \dot{t}_s, \dot{r}_s, \dot{p}_s)$.}
\item{If $i<\gamma$ and $s \in {^i 2}$, then there is $\alpha \in [\alpha_i, \alpha_{i+1})$ such that $(q_{i+1}, \dot{t}_{s^\frown \langle 0 \rangle}$, $\dot{r}_{s^\frown \langle 0 \rangle}, \dot{p}_{s^\frown \langle 0 \rangle})$ and $(q_{i+1}, \dot{t}_{s^\frown \langle 1 \rangle}, \dot{r}_{s^\frown \langle 1 \rangle}, \dot{p}_{s^\frown \langle 1 \rangle})$ decide the statement $``\check{\alpha} \in \dot{f}"$ in opposite ways.}
\item{If $i<\gamma$ and $s \in {^i 2}$, then $(q_{i+1}, \dot{t}_{s^\frown \langle 0 \rangle}, \dot{r}_{s^\frown \langle 0 \rangle}, \dot{p}_{s^\frown \langle 0 \rangle})$ and $(q_{i+1}, \dot{t}_{s^\frown \langle 1 \rangle},$ $\dot{r}_{s^\frown \langle 1 \rangle}, \dot{p}_{s^\frown \langle 1 \rangle})$ both force that $\dot{f} \cap [\alpha_i, \alpha_{i+1}) \not= \emptyset$.}
\end{enumerate}

Suppose for a moment that we have successfully completed this construction. Find $q^* \leq q_\gamma$ such that $q^*$ decides the set $\mathcal{C}_{\alpha_\gamma}$ (this can be done, since $\mathbb{Q}$ is $(\kappa +1)$-strategically closed). Then, for each $s \in {^\gamma 2}$, $(q^*, \dot{t}_s, \dot{r}_s, \dot{p}_s) \in \mathbb{Q}*\mathbb{T}*\mathbb{R}*j(\mathbb{Q})$ and $(q^*, \dot{t}_s, \dot{r}_s, \dot{p}_s) \Vdash ``\alpha_\gamma$ is a limit point of $\dot{f}$". Moreover, if $s,u \in {^\gamma 2}$ and $s \not= u$, then $(q^*, \dot{t}_s, \dot{r}_s, \dot{p}_s)$ and $(q^*, \dot{t}_u, \dot{r}_u, \dot{p}_u)$ force contradictory information about $\dot{f} \cap \alpha_\gamma$, which is forced to be an element of $\mathcal{C}_{\alpha_\gamma}$. But $2^\gamma \geq \kappa$, contradicting the fact that $\mathcal{C}$ is a $\square_{\kappa, <\kappa}$-sequence.

We now turn to the construction. Fix $(q_0, \dot{t}_{\langle \rangle}, \dot{r}_{\langle \rangle}, \dot{p}_{\langle \rangle})$ such that $(q_0, \dot{t}_{\langle \rangle}, \dot{r}_{\langle \rangle}$, $\dot{p}_{\langle \rangle}) \Vdash ``\dot{f}$ is a thread through $\mathcal{C}$", and let $\alpha_0 = 0$. We first consider the successor case. Fix $i <\gamma$ and suppose that $q_i$, $\langle (\dot{t}_s, \dot{r}_s, \dot{p}_s) \mid s\in {^i 2} \rangle$, and $\alpha_i$ have been defined Enumerate ${^i 2}$ as $\langle s_j \mid j< 2^i \rangle$, noting that $2^i<\kappa$. Now, using the $\kappa$-closure of $\mathbb{Q}$, $\mathbb{S}$, $\mathbb{R}$, and $j(\mathbb{Q})$, the density of $\mathbb{S}$ in $\mathbb{Q}*\mathbb{T}$, the subclaim, and the fact that $(q_0, \dot{t}_{\langle \rangle}, \dot{r}_{\langle \rangle}, \dot{p}_{\langle \rangle}) \Vdash ``\dot{f}$ is unbounded in $\kappa^+$", it is straightforward to construct $\langle q^i_j \mid j <2^i \rangle$ and $\langle ((\dot{t}^*_{{s_j}^\frown \langle 0 \rangle}, \dot{r}^*_{{s_j}^\frown \langle 0 \rangle}, \dot{p}^*_{{s_j}^\frown \langle 0 \rangle}), (\dot{t}^*_{{s_j}^\frown \langle 1 \rangle}, \dot{r}^*_{{s_j}^\frown \langle 1 \rangle}, \dot{p}^*_{{s_j}^\frown \langle 1 \rangle})) \mid j<2^i \rangle$ such that:
\begin{itemize}
\item{$\langle q^i_j \mid j <2^i \rangle$ is a decreasing sequence of conditions from $\mathbb{Q}$ below $q_i$.}
\item{For all $j<2^i$, $(q^i_j, \dot{t}^*_{{s_j}^\frown \langle 0 \rangle}, \dot{r}^*_{{s_j}^\frown \langle 0 \rangle}, \dot{p}^*_{{s_j}^\frown \langle 0 \rangle})$, $(q^i_j, \dot{t}^*_{{s_j}^\frown \langle 1 \rangle}, \dot{r}^*_{{s_j}^\frown \langle 1 \rangle}, \dot{p}^*_{{s_j}^\frown \langle 1 \rangle}) \leq$ $(q^i_j, \dot{t}_s, \dot{r}_s, \dot{p}_s)$ are both in $\mathbb{S}*\mathbb{R}*j(\mathbb{Q})$ and both force that $\dot{f} \cap [\alpha_i, \beta^{q^i_j})\not= \emptyset$.}
\item{For all $j<2^i$, there is $\alpha \in [\alpha_i, \beta^{q^i_j})$ such that $(q^i_j, \dot{t}^*_{{s_j}^\frown \langle 0 \rangle}, \dot{r}^*_{{s_j}^\frown \langle 0 \rangle}, \dot{p}^*_{{s_j}^\frown \langle 0 \rangle})$ and $(q^i_j, \dot{t}^*_{{s_j}^\frown \langle 1 \rangle}, \dot{r}^*_{{s_j}^\frown \langle 1 \rangle}, \dot{p}^*_{{s_j}^\frown \langle 1 \rangle})$ decide the statement $``\check{\alpha} \in \dot{f}"$ in opposite ways.}
\end{itemize}
Now let $\xi=\sup (\{\delta \mid$ for some $j<2^i$ and $\ell \in \{0,1\}$, $q^i_j \Vdash ``\mathrm{otp}(\dot{t}^*_{{s_j}^\frown \langle \ell \rangle})=\check{\delta}" \})$. Since $2^i <\kappa$, we know that $\xi<\kappa$. Let $\alpha_{i+1}=\sup(\{\beta^{q^i_j} \mid j<2^i\})$. Let $E_{i+1}$ be a club in $\alpha_{i+1}$ of order type cf($\alpha_{i+1})<\kappa$. Define $q_{i+1}$ to be a lower bound of $\langle q^i_j \mid j <2^i \rangle$ by letting $\beta^{q_{i+1}}=\alpha_{i+1}$ and, for all $k<\kappa$, \[D^{q_{i+1}}(k,\alpha_{i+1})=
\begin{cases}
\emptyset  & \text{if } k<\xi+1 \\
E_{i+1} \cup \bigcup_{\beta \in E_{i+1}} D^{q_{i+1}}(k,\beta) & \text{if } k \geq \xi+1
\end{cases}.\] Finally, for all $j<2^i$ and $\ell \in \{0,1\}$, let $(\dot{t}_{{s_j}^\frown \langle \ell \rangle}, \dot{r}_{{s_j}^\frown \langle \ell \rangle}, \dot{p}_{{s_j}^\frown \langle \ell \rangle})$ be such that $q_{i+1} \Vdash ``(\dot{t}_{{s_j}^\frown \langle \ell \rangle}, \dot{r}_{{s_j}^\frown \langle \ell \rangle}, \dot{p}_{{s_j}^\frown \langle \ell \rangle})$ = $(\dot{t}^*_{{s_j}^\frown \langle \ell \rangle} \cup \{\alpha_{i+1}\}, \dot{r}^*_{{s_j}^\frown \langle \ell \rangle}, \dot{p}^*_{{s_j}^\frown \langle \ell \rangle})"$.

Next, suppose that $i\leq \gamma$ is a limit ordinal and that $\langle q_j \mid j<i \rangle$, $\langle (\dot{t}_s, \dot{r}_s, \dot{p}_s) \mid s \in {^{< i} 2} \rangle$ and $\langle \alpha_j \mid j < i \rangle$ have been defined. Let $\xi=\sup (\{\delta \mid$ for some $s\in {^{<i} 2}$, $q_{|s|} \Vdash ``\mathrm{otp}(\dot{t}_s)=\check{\delta}" \})$. Since $2^{<i}<\kappa$, we know that $\xi<\kappa$. Let $\alpha_i = \sup(\{\beta^{q_j} \mid j<i\})$ and let $E_i$ be a club in $\alpha_i$ of order type cf($\alpha_i)<\kappa$. Define $q_i$ to be a lower bound for $\langle q_j \mid j<i \rangle$ by letting $\beta^{q_i}=\alpha_i$ and \[D^{q_i}(k,\alpha_i)=
\begin{cases}
\emptyset  & \text{if } k<\xi+1 \\
E_{i} \cup \bigcup_{\beta \in E_{i}} D^{q_{i}}(k,\beta) & \text{if } k \geq \xi+1
\end{cases}.\] Finally, for all $s\in {^i 2}$, let $\dot{t}_s$ be such that $q_i \Vdash ``\dot{t}_s = \bigcup_{j<i}\dot{t}_{s\restriction j} \cup \{\alpha_i\}"$ and let $(\dot{r}_s, \dot{p}_s)$ be forced by $(q_i, \dot{t}_s)$ to be a lower bound for $\langle (\dot{r}_{s\restriction j}, \dot{p}_{s\restriction j}) \mid j<i \rangle$. This is possible, since $\mathbb{R}*j(\mathbb{Q})$ is $\kappa$-closed. It is easily verified that this construction satisfies conditions 1-5 above.

But now we have shown that $F$, which threads $\mathcal{C}$, is in $V[G*H]$, contradicting the fact that $\mathcal{C}$ is a $\square_{\kappa, <\kappa}$-sequence in $V[G*H]$. Thus, $\square_{\kappa, <\kappa}$ fails in $V[G*H]$. 
\end{proof}

Note that, if $\kappa$ is supercompact and $\lambda > \kappa$ is measurable, we can also obtain a model in which there is a transitive, normal, uniform $\kappa$-covering matrix for $\kappa^+$ but $\square_{\kappa, <\kappa}$ fails by first making the supercompactness of $\kappa$ indestructible under $\kappa$-directed closed forcing and then forcing with $\mathbb{Q}$. We conjecture that we can obtain such a model for all regular, uncountable $\kappa$ but do not have a proof when $\kappa$ is inaccessible but not supercompact.

We now investigate counterexamples to $\mathrm{CP}(\mathcal{D})$ and $\mathrm{S}(\mathcal{D})$ for more general shapes of covering matrices. Recall the following definitions:

\begin{definition}
Let $\kappa$ be an infinite cardinal. 
 \begin{enumerate}
  \item{$\overrightarrow{C} = \langle C_\alpha \mid \alpha \in \mathrm{lim}(\kappa) \rangle$ is a {\em coherent sequence} if, for all $\alpha, \beta \in \mathrm{lim}(\kappa)$, 
  \begin{enumerate}
   \item{$C_\alpha$ is a club in $\alpha$.}
   \item{If $\alpha \in C'_\beta$, then $C_\alpha = C_\beta \cap \alpha$.}
  \end{enumerate}}
  \item{Let $\overrightarrow{C} = \langle C_\alpha \mid \alpha \in \mathrm{lim}(\kappa) \rangle$ be a coherent sequence. If $D$ is a club in $\kappa$, then $D$ is a {\em thread} through $\overrightarrow{C}$ if, for every $\alpha \in D'$, $D\cap \alpha = C_\alpha$.}
  \item{$\overrightarrow{C} = \langle C_\alpha \mid \alpha \in \mathrm{lim}(\kappa) \rangle$ is a $\square(\kappa)$-sequence if it is a coherent sequence that has no thread. We say that $\square(\kappa)$ holds if there is a $\square(\kappa)$-sequence.}
 \end{enumerate}
\end{definition}

Let $\theta <  \lambda$ be regular, infinite cardinals and suppose that $\square(\lambda)$ holds. Fix a $\square(\lambda)$-sequence $\overrightarrow{C}$. Arrange so that $C_\alpha$ is defined for all $\alpha < \lambda$ by letting $C_{\alpha+1} = \{\alpha \}$. The definitions of the following functions are due to Todorcevic \cite{todorcevic}: First, define $\Lambda_\theta : [\lambda]^2 \rightarrow \lambda$ by 
\[\Lambda_\theta (\alpha, \beta) = \max \{ \xi \in C_\beta \cap (\alpha + 1) \mid \theta \mbox{ divides } \mathrm{otp}(C_\beta \cap \xi) \} \]
Next, let $\rho_\theta : [\lambda]^2 \rightarrow \theta$ be defined recursively by
\begin{multline*}
\rho_\theta (\alpha, \beta) = \sup \{\mathrm{otp}(C_\beta \cap [\Lambda_\theta(\alpha, \beta),\alpha)), \rho_\theta(\alpha, \min(C_\beta \setminus \alpha)), \\ \rho_\theta(\xi, \alpha) \mid \xi \in C_\beta \cap [\Lambda_\theta(\alpha, \beta), \alpha)\}
\end{multline*}

Proofs of the following lemmas can be found in \cite{todorcevic}.

\begin{lemma}
\label{rho}
Let $\alpha < \beta < \gamma < \lambda$. 
\begin{enumerate}
 \item{$\rho_\theta(\alpha, \gamma) \leq \max \{ \rho_\theta(\alpha, \beta), \rho_\theta(\beta, \gamma) \}$.}
 \item{$\rho_\theta(\alpha, \beta) \leq \max \{ \rho_\theta(\alpha, \gamma), \rho_\theta(\beta, \gamma) \}$.}
\end{enumerate}
\end{lemma}

\begin{lemma}
\label{closed}
Let $i<\theta$ and $\beta < \lambda$. $\{ \alpha < \beta \mid \rho_\theta(\alpha, \beta) \leq i \}$ is closed.
\end{lemma}

Define a covering matrix $\mathcal{D} = \{D(i,\beta) \mid i<\theta, \beta < \lambda \}$ by letting $D(i,\beta) = \{ \alpha < \beta \mid \rho_\theta(\alpha, \beta) \leq i \}$. It is clear that $\mathcal{D}$ satisfies conditions 1 and 2 in the definition of a covering matrix. Part 1 of Lemma \ref{rho} implies that $\mathcal{D}$ is transitive. In fact, together with part 2 of the same lemma, it implies a stronger coherence property, namely that if $i<\theta$, $\alpha < \beta < \lambda$, and $\alpha \in D(i,\beta)$, then $D(i,\alpha) = D(i,\beta) \cap \alpha$. Lemma \ref{closed} implies that $\mathcal{D}$ is closed. We now show, however, that in general $\mathcal{D}$ is not uniform. First, we make the following definition:

\begin{definition}
\begin{enumerate}
\item{Let $A$ be a set of ordinals and let $\mu$ be an infinite cardinal. $A^{[\mu]} = \{\alpha \in A \mid \mu \mbox{ divides } A \cap \alpha \}$.}
\item{Let $\mu < \kappa$ be infinite regular cardinals. $\overrightarrow{C}$ is a $\square^\mu(\kappa)$-sequence if $\overrightarrow{C}$ is a $\square(\kappa)$-sequence and $\{\alpha \in S^\kappa_\mu \mid C^{[\mu]}_\alpha \mbox{ is bounded below } \alpha \}$ is stationary. $\square^\mu(\kappa)$ is the statement that a $\square^\mu(\kappa)$-sequence exists.}
\end{enumerate}
\end{definition}

\begin{lemma}
 Let $\beta < \lambda$ be such that $\mathrm{cf}(\beta) = \theta$ and $C^{[\theta]}_\beta$ is bounded below $\beta$. Then, for every $i<\theta$, $D(i,\beta)$ is bounded below $\beta$.
\end{lemma}

\begin{proof}
Let $\xi = \max(C^{[\theta]}_\beta)$. Since $\xi < \beta$ and $\mathrm{cf}(\beta) = \theta$, $\mathrm{otp}(C_\beta \setminus (\xi + 1)) = \theta$. Enumerate $C_\beta \setminus (\xi +1)$ in increasing order as $\langle \beta_i \mid i < \theta \rangle$. Now, if $i<\theta$ and $\beta_i < \alpha < \beta$, then $\Lambda_\theta(\alpha, \beta) = \xi$ and $\mathrm{otp}(C_\beta \cap [\xi, \alpha)) > i$, so $\rho_\theta(\alpha, \beta)>i$. Thus, $D(i,\beta) \subseteq (\beta_i +1)$. 
\end{proof}

\begin{lemma}
 Suppose $\overrightarrow{C}$ is a $\square^\theta(\lambda)$-sequence. Then $\mathrm{CP}(\mathcal{D}$) and $\mathrm{S}(\mathcal{D}$) both fail.
\end{lemma}

\begin{proof}
 Suppose for sake of contradiction that $\mathrm{CP}(\mathcal{D}$) holds and is witnessed by an unbounded $T\subseteq \lambda$. Since $\overrightarrow{C}$ is a $\square^\theta(\lambda)$-sequence, we can find $\alpha \in T'$ such that $C^{[\theta]}_\alpha$ is bounded below $\alpha$. Let $X\in [T]^\theta$ be an unbounded subset of $\alpha$. Find $i<\theta$ and $\beta < \lambda$ such that $X\subseteq D(i,\beta)$. Since $\mathcal{D}$ is closed, $\alpha \in D(i,\beta)$, so $D(i,\alpha)=D(i,\beta) \cap \alpha$ and $X \subseteq D(i,\alpha)$. This is a contradiction, as the previous lemma implies that $D(i,\alpha)$ is bounded below $\alpha$. Thus, $\mathrm{CP}(\mathcal{D}$) fails. Since $\mathcal{D}$ is transitive, this means that $\mathrm{S}(\mathcal{D}$) fails as well. 
\end{proof}

The question now naturally arises whether $\square^\theta(\lambda)$ is a strictly stronger assumption than $\square(\lambda)$. This and related questions are addressed in the remainder of this paper.

\section{Squares}

\begin{definition}
Let $A\subseteq \kappa$ and let $\overrightarrow{C}$ be a $\square(\kappa)$-sequence. $\overrightarrow{C}$ {\em avoids} $A$ if, for every $\alpha \in \mathrm{lim}(\kappa)$, $C'_\alpha \cap A = \emptyset$.
\end{definition}

It is well known that if $\square_\kappa$ holds and $S\subseteq \kappa^+$ is stationary, then there is a stationary $T\subseteq S$ and a $\square_\kappa$-sequence that avoids $T$. We would like to know to what extent similar phenomena occur in connection with $\square(\kappa)$. The following proposition, whose proof we include for completeness, provides some information in this direction by showing that, if $\kappa$ is regular and $S\subseteq \kappa$ is stationary, then every $\square(\kappa)$-sequence must, in a certain sense, avoid a pair of stationary subsets of $S$.

\begin{proposition}
 Let $\kappa >\omega_1$ be a regular cardinal, and let $\overrightarrow{C}=\langle C_\alpha \mid \alpha \in \mathrm{lim}(\kappa)\rangle$ be a coherent sequence. Then the following are equivalent:
\begin{enumerate}
 \item {$\overrightarrow{C}$ is a $\square(\kappa)$-sequence.}
 \item{For every stationary $T\subseteq \kappa$, there are stationary $S_0, S_1\subseteq T$ such that for every $\alpha \in \mathrm{lim}(\kappa)$, $C'_\alpha \cap S_0 = \emptyset$ or $C'_\alpha \cap S_1 = \emptyset$.}
\end{enumerate}
\end{proposition}

\begin{proof}
 First, suppose that there are stationary sets $S_0$, $S_1 \subseteq \kappa$ such that for every $\alpha \in \mathrm{lim}(\kappa)$, $C'_\alpha \cap S_0 = \emptyset$ or $C'_\alpha \cap S_1 = \emptyset$ and suppose for sake of contradiction that there is a club $D$ in $\kappa$ that threads $\overrightarrow{C}$. Since $S_0$ and $S_1$ are stationary, there are $\alpha<\beta<\gamma$ in $D'$ such that $\alpha \in S_0$ and $\beta \in S_1$. Since $D\cap \gamma = C_\gamma$, we have $\alpha, \beta \in C'_\gamma$, contradicting the fact that $C'_\gamma \cap S_0$ or $C'_\gamma \cap S_1$ is empty. Thus, $\overrightarrow{C}$ is a $\square(\kappa)$-sequence.

Next, suppose that $\overrightarrow{C}$ is a $\square(\kappa)$-sequence and that $T\subseteq \kappa$ is stationary. There are two cases:

{\bf Case 1:} There is $\alpha_0 < \kappa$ such that $\{\alpha \in T \mid C'_\alpha \cap (\alpha_0, \kappa) \not= \emptyset \}$ is nonstationary. Let $T^* = \{\alpha \in T\setminus (\alpha_0+1) \mid C'_\alpha \cap (\alpha_0, \kappa) = \emptyset \}$ and let $T^*=S_0 \cup S_1$ be any partition of $T^*$ into disjoint stationary sets. Then $S_0$ and $S_1$ are as desired.

{\bf Case 2:} For all $\alpha_0 < \kappa$, $\{\alpha \in T \mid C'_\alpha \cap (\alpha_0, \kappa) \not= \emptyset \}$ is stationary. For $\alpha \in \mathrm{lim}(\kappa)$, let $S^0_\alpha = \{\beta \in T\setminus (\alpha +1) \mid \alpha \not\in C'_\beta \}$ and let $S^1_\alpha = \{\beta \in T\setminus (\alpha +1) \mid \alpha \in C'_\beta \}$. 
\begin{claim}
 There is $\alpha \in \mathrm{lim}(\kappa)$ such that $S^0_\alpha$ and $S^1_\alpha$ are both stationary.
\end{claim}
\begin{proof}
 Suppose this is not the case. Let $A$ be the set of $\alpha \in \mathrm{lim}(\kappa)$ such that $S^0_\alpha$ is nonstationary. We claim that $A$ is unbounded in $\kappa$. To show this, fix $\alpha_0 <\kappa$. By Fodor's Lemma, we can fix an $\alpha \geq \alpha_0$ and a stationary $T^* \subseteq T$ such that if $\beta \in T^*$, then $\alpha = \min (C'_\beta \setminus \alpha_0)$. Then $\alpha \in A$. Now let $\alpha < \alpha'$ be elements of $A$, and let $D_\alpha$ and $D_{\alpha'}$ be clubs in $\kappa$ disjoint from $S^0_\alpha$ and $S^0_{\alpha'}$, respectively. Fix $\beta \in D_\alpha \cap D_{\alpha'} \cap T$. Then $C_\alpha = C_\beta \cap \alpha$ and $C_{\alpha'} = C_\beta \cap \alpha'$, so $C_\alpha = C_{\alpha'} \cap \alpha$. Thus, $D = \bigcup_{\alpha \in A}C_\alpha$ is a thread through $\overrightarrow{C}$, contradicting the fact that $\overrightarrow{C}$ is a $\square(\kappa)$-sequence. 
\end{proof}

Now let $\alpha$ be such that $S^0_\alpha$ and $S^1_\alpha$ are both stationary. Let $S_0 = S^0_\alpha$ and $S_1 = S^1_\alpha$. It is routine to check that $S_0$ and $S_1$ are as desired. 
 
\end{proof}

The preceding observations lead us to make the following definition.

\begin{definition}
 Let $\kappa$ be a regular, uncountable cardinal, and let $S\subseteq \kappa$. $\overrightarrow{C}$ is a $\square(\kappa, S)$-sequence if $\overrightarrow{C}$ is a $\square(\kappa)$-sequence and $\overrightarrow{C}$ avoids $S$. $\square(\kappa, S)$ is the statement that a $\square(\kappa, S)$-sequence exists.
\end{definition}

We start our investigation of these intermediate square principles with the following simple observation.

\begin{proposition}
\label{prop33}
 Let $\mu<\kappa$ be infinite regular cardinals. The following are equivalent:
 \begin{enumerate}
   \item{$\square^\mu(\kappa)$}
   \item{There is a $\square(\kappa)$-sequence $\overrightarrow{C}$ such that $\{\alpha < \kappa \mid \mathrm{otp}(C_\alpha) = \mu \}$ is stationary.}
 \end{enumerate}
\end{proposition}

\begin{proof}
 A $\square(\kappa)$-sequence as in 2. is clearly a $\square^\mu(\kappa)$ sequence. For the other direction, let $\overrightarrow{D}$ be a $\square^\mu(\kappa)$-sequence, and let $T = \{\alpha < \kappa \mid D^{[\mu]}_\alpha \mbox{ is bounded below } \alpha \}$. Form $\overrightarrow{C}$ as follows. If $\alpha \in \mathrm{lim}(\kappa) \setminus T$, let $C_\alpha = D^{[\mu]}_\alpha$. If $\alpha \in T$, let $C_\alpha = D_\alpha \setminus (\max(D^{[\mu]}_\alpha) + 1)$. It is easy to verify that $\overrightarrow{C}$ is a $\square(\kappa)$-sequence and, if $\alpha \in T \cap S^\kappa_\mu$, then $\mathrm{otp}(C_\alpha) = \mu$. 
\end{proof}

The remainder of the paper investigates the implications and non-implications that exist among the traditional square properties and those of the form $\square^\mu(\kappa)$ and $\square(\kappa, S)$. A diagram illustrating the complete picture when $\kappa = \omega_2$ can be found at the end of the paper.

\begin{proposition}
 Let $\mu<\kappa$ be infinite, regular cardinals. 
\begin{enumerate}
 \item {If $\square^\mu(\kappa)$ holds, then there is a stationary $S\subseteq S^\kappa_\mu$ such that $\square(\kappa, S)$ holds.}
 \item {If there is a stationary $S\subseteq S^\kappa_\omega$ such that $\square(\kappa, S)$ holds, then $\square^\omega(\kappa)$ holds.} 
\end{enumerate}

\end{proposition}

\begin{proof}
 First, suppose that $\overrightarrow{C}$ is a $\square^\mu(\kappa)$-sequence. Let $T=\{\alpha < \kappa \mid C^{[\mu]}_\alpha$ is bounded below $\alpha \}$. It is easily seen that the construction from the proof of Proposition \ref{prop33} yields a $\square(\kappa, S)$-sequence, where $S=T\cap S^\kappa_\mu$.

Next, suppose that $S\subseteq S^\kappa_\omega$ is stationary and that $\overrightarrow{D}$ is a $\square(\kappa, S)$-sequence. For $\alpha <\kappa$, define $C_\alpha$ as follows: If $\alpha \not\in S$, let $C_\alpha=D_\alpha$. If $\alpha \in S$, let $C_\alpha$ be any set of order type $\omega$ unbounded in $\alpha$. Since $\overrightarrow{D}$ avoids $S$, this does not interfere with the coherence of the sequence, so $\overrightarrow{C}$ is a $\square^\omega(\kappa)$-sequence. 
\end{proof}

\begin{proposition}
 Let $\mu<\nu<\kappa$ be infinite, regular cardinals. If $\square^\nu(\kappa)$ holds, then $\square^\mu(\kappa)$ holds.
\end{proposition}

\begin{proof}
 Assume $\square^\nu(\kappa)$ holds, and fix a $\square(\kappa)$-sequence $\overrightarrow{C}$ such that $T_0 = \{\alpha < \kappa \mid \mathrm{otp}(C_\alpha) = \nu \}$ is stationary. We claim that $T_1 = \{\alpha \in S^\kappa_\mu \mid \mathrm{otp}(C_\alpha) < \nu \}$ is also stationary. To see this, let $E$ be club in $\kappa$. Let $\beta \in E'\cap T_0$. Then $E\cap C_\beta$ is club in $\beta$. Let $\alpha \in (E\cap C_\beta)' \cap S^\kappa_\mu$. Then, since $C_\alpha = C_\beta \cap \alpha$, $\mathrm{otp}(C_\alpha)<\nu$, so $\alpha \in E\cap T_1$.
 
 We can now apply Fodor's Lemma to $T_1 \setminus \nu$ to find a $\gamma < \nu$ and a stationary $S\subseteq T_1$ such that $\alpha \in S$ implies $\mathrm{otp}(C_\alpha) = \gamma$. Let $\langle \gamma_\xi \mid \xi < \mu \rangle$ enumerate a club in $\gamma$ with each $\gamma_\xi$ a limit ordinal (this will not be possible if $\mu = \omega$ and $\gamma$ is not a limit of limit ordinals, but in that case $\overrightarrow{C}$ is already a $\square^\mu(\kappa)$-sequence).

We now define a $\square^\mu(\kappa)$-sequence $\overrightarrow{D}$ (in fact, $\overrightarrow{D}$ will also be a $\square(\kappa, S)$-sequence). First, if $\alpha \in S$, let $D_\alpha = \{\beta \in C'_\alpha \mid \mbox{for some } \nu<\mu, \mbox{ otp}(C_\alpha \cap \beta)=\gamma_\nu \}$. Next, if $\alpha \in D'_{\alpha'}$ for some $\alpha' \in S$, let $D_\alpha = D_{\alpha'} \cap \alpha$. Note that this is well-defined. If $\alpha \in C'_{\alpha'}$ for some $\alpha' \in S$ but, for all $\beta \in S$, $\alpha \not\in D'_\beta$, then let $D_\alpha = C_\alpha \setminus \max(D_{\alpha'} \cap \alpha)$. Note again that this is well-defined.

If there is $\alpha' \in S$ such that $\alpha' \in C'_\alpha$ (note that such an $\alpha'$ must be unique), then let $D_\alpha = C_\alpha \setminus \alpha'$. In all other cases, let $D_\alpha = C_\alpha$. It is now easy to verify that $\overrightarrow{D}$ is a $\square(\kappa)$-sequence and, since $\alpha \in S$ implies that $\mathrm{otp}(D_\alpha)=\mu$, that it is in fact a $\square^\mu(\kappa)$-sequence.
\end{proof}

The following corollary is now immediate.

\begin{corollary}
 \begin{enumerate}
  \item{Let $\mu < \nu < \kappa$ be infinite, regular cardinals. If $\square^\nu(\kappa)$ holds, then there is a stationary $S\subseteq S^\kappa_\mu$ such that $\square(\kappa, S)$ holds.}
  \item{Let $\mu \leq \kappa$ be infinite, regular cardinals, with $\mu$ regular. If $\square_\kappa$ holds, then $\square^\mu(\kappa^+)$ holds.}
 \end{enumerate}
\end{corollary}

We now show that the above implications are generally not reversible. We begin by recalling the definition of the forcing poset that adds a $\square(\kappa)$ sequence by specifying its initial segments.

\begin{definition}
 Let $\kappa$ be a regular cardinal. $\mathbb{Q}(\kappa)$ is the partial order whose elements are of the form $q=\langle C^q_\alpha \mid \alpha \leq \beta^q \rangle$, where
\begin{enumerate}
 \item {$\beta^q<\kappa$.}
 \item {For all $\alpha \leq \beta^q$, $C^q_\alpha$ is a club in $\alpha$.}
 \item {For all $\alpha < \alpha' \leq \beta^q$, if $\alpha \in C'_{\alpha'}$, then $C_\alpha = C_{\alpha'} \cap \alpha$.}
\end{enumerate}
$p\leq q$ if and only if $p$ end-extends $q$, i.e. $\beta^p \geq \beta^q$ and, for all $\alpha \leq \beta^q$, $C^p_\alpha = C^q_\alpha$.
\end{definition}

\begin{proposition}
$\mathbb{Q}(\kappa)$ is $\kappa$-strategically closed.
\end{proposition}

\begin{proof}
 We specify a winning strategy for Player II in $G_\kappa(\mathbb{Q}(\kappa))$. First, let $q_0 = \emptyset$. Let $0 < \alpha < \kappa$ be even and suppose that $\langle q_\delta \mid \delta < \alpha \rangle$ has already been played. We specify Player II's next move, $q_\alpha$. Let $E_\alpha = \{\beta^{q_\delta} \mid \delta < \alpha \mbox{ is even} \}$ and suppose that we have satisfied the following inductive hypotheses:
\begin{enumerate}
 \item {$E_\alpha$ is closed below its supremum.}
 \item {For all even ordinals $\delta < \xi < \alpha$, $\beta^{q_\delta}<\beta^{q_{\xi}}$ and $\beta^{q_\delta} \in (C^{q_\xi}_{\beta^{q_{\xi}}})'$.}
\end{enumerate}

First, suppose that $\alpha$ is a successor ordinal. Since it is even, it is in fact a double successor. Let $\alpha = \alpha'+1 = \alpha''+2$. In this case, let $\beta^{q_\alpha}=\beta^{q_{\alpha'}} + \omega$. For limit ordinals $\zeta<\beta^{q_\alpha}$, let $C^{q_\alpha}_\zeta = C^{q_{\alpha'}}_\zeta$, and let \[C^{q_\alpha}_{\beta^{q_\alpha}} = C^{q_\alpha}_{\beta^{q_{\alpha''}}} \cup \{\beta^{q_{\alpha''}}\} \cup \{\beta^{q_{\alpha'}}+n \mid n < \omega \}.\]

Next, suppose that $\alpha$ is a limit ordinal. Let $\beta^{q_\alpha}=\sup(\{\beta^{q_\delta} \mid \delta < \alpha\})$. For limit ordinals $\zeta < \beta^{q_\alpha}$, find $\delta < \alpha$ such that $\zeta \leq \beta^{q_\delta}$ and let $C^{q_\alpha}_\zeta = C^{q_\delta}_\zeta$. Note that this is well-defined. Let \[C^{q_\alpha}_{\beta^{q_\alpha}}= \bigcup_{\zeta \in E_\alpha} C^{q_\alpha}_\zeta.\] By our inductive hypotheses, this is a club in $\beta^{q_\alpha}$ and satisfies the coherence requirements.

It is clear that this procedure produces a valid condition $q_\alpha \in \mathbb{Q}(\kappa)$ that is a lower bound for $\langle q_\delta \mid \delta < \alpha \rangle$ and maintains the inductive hypotheses. Thus, $\mathbb{Q}(\kappa)$ is $\kappa$-strategically  closed. 
\end{proof}

An argument similar to the proof of the previous proposition shows that, for every $\alpha<\kappa$, the set $\{q \mid \beta^q \geq \alpha \}$ is dense in $\mathbb{Q}(\kappa)$.

\begin{corollary}
 Forcing with $\mathbb{Q}(\kappa)$ preserves all cardinals $\leq \kappa$ and adds a coherent sequence $\langle C_\alpha \mid \alpha < \kappa \rangle$. In addition, if $\kappa^{<\kappa}=\kappa$, then all cardinals are preserved.
\end{corollary}

\begin{proof}
 Let $G$ be $\mathbb{Q}(\kappa)$-generic over $V$. Since $\mathbb{Q}(\kappa)$ is $\kappa$-strategically closed, it doesn't add any $<\kappa$-sequences of ordinals and hence preserves all cardinals $\leq \kappa$. Since $\{q \mid \beta^q \geq \alpha \}$ is dense in $\mathbb{Q}(\kappa)$ for every $\alpha < \kappa$, we can define $C_\alpha = C^q_\alpha$, where $q \in G$ and $\beta^q \geq \alpha$. It is clear that $\overrightarrow{C}=\langle C_\alpha \mid \alpha < \kappa \rangle$ is well-defined and a coherent sequence. Finally, if $\kappa^{<\kappa}$, then $|\mathbb{Q}(\kappa)| = \kappa$. Thus, $\mathbb{Q}(\kappa)$ has the $\kappa^{+}$-c.c. and preserves all cardinals $\geq \kappa^{+}$. 
\end{proof}

\begin{lemma}
 If $\mu<\kappa$ are regular cardinals and $G$ is $\mathbb{Q}(\kappa)$-generic over $V$, then the coherent sequence $\overrightarrow{C}$ added by $G$ is a $\square^\mu(\kappa)$-sequence.
\end{lemma}

\begin{proof}
 Let $S = \{\alpha < \kappa \mid \mathrm{otp}(C_\alpha) = \mu \}$. It suffices to show that, in $V[G]$, $S$ is stationary in $\kappa$. Note that this implies that $\overrightarrow{C}$ doesn't have a thread, since any club in $\kappa$ must meet $S$ in two points.
  
 Work in $V$, let $\dot{D}$ be a $\mathbb{Q}(\kappa)$-name forced by the empty condition to be a club in $\kappa$, let $\dot{S}$ be a $\mathbb{Q}(\kappa)$-name for $S$, and let $q \in \mathbb{Q}(\kappa)$. We will find $p \leq q$ such that $p \Vdash \dot{D} \cap \dot{S} \not= \emptyset$.

We construct $\langle q_\alpha \mid \alpha \leq \mu \rangle$, a decreasing sequence of conditions from $\mathbb{Q}(\kappa)$ such that for every $\alpha \leq \mu$,
\begin{enumerate}
 \item {$\mathrm{otp}(C^{q_\alpha}_{\beta^{q_\alpha}}) \leq \mu$.}
 \item {$E_\alpha = \{\beta^{q_\delta} \mid \delta < \alpha \}$ is closed below its supremum.}
 \item {For all $\delta<\alpha$, $\beta^{q_\delta}<\beta^{q_\alpha}$ and $\beta^{q_\delta} \in (C^{q_\alpha}_{\beta^{q_\alpha}})'$.}
 \item {If $\alpha<\mu$, then $q_{\alpha+1} \Vdash \dot{D} \cap (\beta^{q_\alpha}, \beta^{q_{\alpha+1}}) \not= \emptyset$.}
\end{enumerate}

To carry out this construction, we first let $\beta^{q_0}=\beta^q + \omega$. For limit ordinals $\zeta \leq \beta^q$, let $C^{q_0}_\zeta = C^q_\zeta$. Let $C^{q_0}_{\beta^{q_0}} = \{\beta^q + n \mid n < \omega \}$. Next, suppose that $\alpha = \alpha' + 1$ and that we have already constructed $\langle q_\delta \mid \delta \leq \alpha' \rangle$. Find $q^*_\alpha \leq q_{\alpha'}$ and $\xi_\alpha > \beta^{q_{\alpha'}}$ such that $q^*_\alpha \Vdash \xi_\alpha \in \dot{D}$. Find $q^{**}_\alpha \leq q^*_\alpha$ such that $\beta^{q^{**}_\alpha} \geq \xi_\alpha$. Let $\beta^{q_\alpha}=\beta^{q^{**}_\alpha} + \omega$. For limit $\zeta \leq \beta^{q^{**}_\alpha}$, let $C^{q_\alpha}_\zeta = C^{q^{**}_\alpha}_\zeta$. Finally, let $C^{q_\alpha}_{\beta^{q_\alpha}} = C^{q_\alpha}_{\beta^{q_{\alpha'}}} \cup \{\beta^{q_{\alpha'}}\} \cup \{\beta^{q^{**}_\alpha} + n \mid n < \omega \}$.

Now suppose that $\alpha < \mu$ is a limit ordinal and we have constructed $\langle q_\delta \mid \delta < \alpha \rangle$. Let $\beta^{q_\alpha} = \sup(\{\beta^{q_\delta} \mid \delta <\alpha \})$. For limit ordinals $\zeta < \beta^{q_\alpha}$, find $\delta < \alpha$ such that $\beta^{q_\delta} \geq \zeta$ and let $C^{q_\alpha}_\zeta = C^{q_\delta}_\zeta$. Let \[C^{q_\alpha}_{\beta^{q_\alpha}} = \bigcup_{\zeta \in E_\alpha} C^{q_\alpha}_\zeta.\] 

It is clear that this construction satisfies requirements 1-4 above. Let $p = q_\mu$. We have arranged so that cf($\beta^p)=\mu$ and $\mathrm{otp}(C^p_{\beta^p})=\mu$. We have also arranged that, for every $\zeta < \beta^p$, $p \Vdash ``\dot{D} \cap (\zeta, \beta^p) \not= \emptyset"$. Thus, since $\dot{D}$ is forced by the empty condition to be a club, $p \Vdash \beta^p \in \dot{D}$. Thus we have found our desired $p \leq q$ such that $p ``\Vdash \dot{D} \cap \dot{S} \not= \emptyset"$. 
\end{proof}

We now introduce a forcing poset designed to add a thread of order type $\kappa$ through a $\square(\kappa)$-sequence.

\begin{definition}
 Let $\kappa$ be a regular cardinal and let $\overrightarrow{C}$ be a $\square(\kappa)$-sequence. $\mathbb{T}(\overrightarrow{C})$ is the partial order consisting of elements $t$ such that:
\begin{enumerate}
 \item {$t$ is a closed, bounded subset of $\kappa$.}
 \item {For every $\alpha \in t'$, $t\cap \alpha = C_\alpha$.}
\end{enumerate}
We denote the maximum element of a condition $t$ by $\gamma^t$. $s \leq t$ if and only if $s$ end-extends $t$, i.e. $\gamma^s \geq \gamma^t$ and $s \cap (\gamma^t + 1) = t$.

\end{definition}

As was the case with the previously defined $\mathbb{T}_\mathcal{D}$, if $\overrightarrow{C}$ was added by $\mathbb{Q}(\kappa)$, then $\mathbb{T}(\overrightarrow{C})$ is quite nice.

\begin{proposition}
 Let $\kappa$ be a regular cardinal and let $\overrightarrow{C}$ be a $\square(\kappa)$-sequence. For every $\alpha<\kappa$, the set $\{t \mid \gamma^t \geq \alpha \}$ is dense in $\mathbb{T}(\overrightarrow{C})$.
\end{proposition}

\begin{proof}
 If $t\in \mathbb{T}(\overrightarrow{C})$ and $\gamma^t < \alpha$, then $t\cup \{\alpha \} \in \mathbb{T}(\overrightarrow{C})$.
\end{proof}

\begin{proposition}
\label{thread}
 Let $\kappa$ be a regular cardinal. Let $\mathbb{Q}=\mathbb{Q}(\kappa)$, $\dot{\overrightarrow{C}}$ be a $\mathbb{Q}$-name for the $\square(\kappa)$-sequence added by $\mathbb{Q}$, and $\dot{\mathbb{T}}$ be a $\mathbb{Q}$-name for $\mathbb{T}(\dot{\overrightarrow{C}})$. Then $\mathbb{Q}*\dot{\mathbb{T}}$ has a $\kappa$-closed dense subset.
\end{proposition}

\begin{proof}
 Let $\mathbb{S} = \{(q, \dot{t}) \mid q \mbox{ decides the value of } \dot{t} \mbox{ and } q\Vdash ``\beta^q = \gamma^t" \}$. We first show that $\mathbb{S}$ is dense in $\mathbb{Q}*\dot{\mathbb{T}}$. To this end, let $(q_0, \dot{t}_0) \in \mathbb{Q}*\dot{\mathbb{T}}$. Since $\mathbb{Q}$ is $\kappa$-strategically closed, $\dot{t}_0$ is forced to be in the ground model. Find $q \leq q_0$ and $t^*$ such that $q \Vdash ``\dot{t}_0 = \check{t}^*"$. Without loss of generality, we may assume that $\beta^{q} > \mathrm{max}(t^*)$. Let $\dot{t}$ be such that $q \Vdash ``\dot{t} = \dot{t}_0 \cup \{\check{\beta}^{q} \}"$. Then $(q, \dot{t}) \leq (q_0, \dot{t}_0)$ and $(q, \dot{t})\in \mathbb{S}$.

Next, we claim that $\mathbb{S}$ is $\kappa$-closed. Let $\alpha < \kappa$ and let $\langle (q_\delta, \dot{t}_\delta) \mid \delta < \alpha \rangle$ be a decreasing sequence of conditions from $\mathbb{S}$. Without loss of generality, we may assume that, for every $\delta < \alpha$, $\beta^{q_\delta} < \beta^q$. We will construct a lower bound $(q, \dot{t})$. Let $\beta^q = \sup(\{\beta^{q_\delta} \mid \delta < \alpha \})$.  For limit $\zeta < \beta^q$, let $\delta < \alpha$ be such that $\beta^{q_\delta} \geq \zeta$ and set $C^{q}_\zeta = C^{q_\delta}_\zeta$. Let $X = \{\zeta \mid \mbox{for some } \delta < \alpha, q_\delta \Vdash ``\check{\zeta} \in \dot{t}_\alpha" \}$. By our definition of $\mathbb{S}$, $X$ is club in $\beta^q$ and for every $\zeta \in X'$, $X \cap \zeta = C^q_\zeta$. Thus, we can let $C^q_{\beta^q} = X$. Finally, let $\dot{t}$ be such that $q\Vdash ``\dot{t} = \check{X} \cup \{\check{\beta}^q \}"$. $(q,\dot{t})$ is then a lower bound of $\langle (q_\delta, \dot{t}_\delta) \mid \delta < \alpha \rangle$ in $\mathbb{S}$. 
\end{proof}

A key point here, which will be exploited in the proof of the next theorem, is that, for an uncountable cardinal $\kappa$, one can force to add and then thread a $\square(\kappa^+)$-sequence with a two-step iteration which is $\kappa^+$-closed, whereas if one wants to add and thread, for example, a $\square_{\kappa, < \kappa}$-sequence, the best one can do is a two-step iteration which is $\kappa$-closed.

\begin{theorem}
\label{th312}
 Suppose $\mu < \kappa$ are regular cardinals and $\lambda > \kappa$ is a measurable cardinal. Let $G$ be $\mathrm{Coll}(\kappa, <\lambda$)-generic over $V$ and, in $V[G]$, let $H$ be $\mathbb{Q}(\kappa^+)$-generic over $V[G]$. Then, in $V[G*H]$, $\square^\mu(\kappa^+)$ holds and $\square_{\kappa, <\kappa}$ fails.
\end{theorem}

Note that, in $V[G*H]$, $\kappa^{<\kappa} = \kappa$, so $\square^*_\kappa$ holds.

\begin{proof}
 We have already shown that $\square^\mu(\kappa^+)$ holds in any extension by $\mathbb{Q}(\kappa^+)$, so it remains to show that $\square_{\kappa, <\kappa}$ fails. Let $\mathbb{Q}=\mathbb{Q}(\kappa^+)$, and let $\overrightarrow{C}$ be the $\square(\kappa^+)$-sequence added by $H$. In $V[G*H]$, let $\mathbb{T} = \mathbb{T}(\overrightarrow{C})$.

Fix an elementary embedding $j:V \rightarrow M$ with critical point $\lambda$. $j\restriction \mathrm{Coll}(\kappa, <\lambda):\mathrm{Coll}(\kappa, <\lambda) \rightarrow \mathrm{Coll}(\kappa, <j(\lambda))$ is the identity map and, in $V^{\mathrm{Coll}(\kappa, <\lambda)}$, $\mathbb{Q}*\dot{\mathbb{T}}$ has a $\kappa^+$-closed dense subset which has size $\kappa^+$. Thus, we can extend $j$ to a complete embedding of $\mathrm{Coll}(\kappa, <\lambda)*\dot{\mathbb{Q}}*\dot{\mathbb{T}}$ into $\mathrm{Coll}(\kappa, <j(\lambda))$ such that the quotient forcing, $\mathbb{R}$, is $\kappa$-closed. Then, letting $I$ be $\mathbb{T}$-generic over $V[G*H]$ and letting $J$ be $\mathbb{R}$-generic over $V[G*H*I]$, we can further extend $j$ to an elementary embedding $j:V[G] \rightarrow M[G*H*I*J]$.

We would like to extend $j$ still further to have domain $V[G*H]$. This is precisely the reason for introducing the threading poset. In $V[G*H*I*J]$, $j(\mathbb{Q})$ is the forcing poset to add a $\square(j(\lambda))$-sequence. $\langle C_\alpha \mid \alpha < \lambda \rangle$ would be a condition in $j(\mathbb{Q})$ if it had a top element. To arrange this, we define $q^*\in j(\mathbb{Q})$ by letting $\beta^{q^*} = \lambda$, $C^{q^*}_\alpha = C_\alpha$ for all $\alpha < \lambda$, and $C^{q^*}_\lambda = \bigcup I$. Since $\bigcup I$ is a thread through $\overrightarrow{C}$, $q^*$ is a condition in $j(\mathbb{Q})$. Moreover, for every $q\in H$, $j(q) = q \leq q^*$. Thus, if $K$ is $j(\mathbb{Q})$-generic over $V[G*H*I*J]$ and $q^* \in K$, then $j[H] \subseteq K$, so we can extend $j$ to an elementary embedding $j:V[G*H]\rightarrow M[G*H*I*J*K]$.

Now suppose for sake of contradiction that $\overrightarrow{\mathcal{D}} = \langle \mathcal{D}_\alpha \mid \alpha < \lambda \rangle$ is a $\square_{\kappa, <\kappa}$-sequence in $V[G*H]$. For each $\alpha < \lambda$, $j(\mathcal{D}_\alpha) = \mathcal{D}_\alpha$. Let $j(\overrightarrow{\mathcal{D}}) = \langle \mathcal{D}_\alpha \mid \alpha < j(\lambda) \rangle$. $j(\overrightarrow{\mathcal{D}})$ is a $\square_{\kappa, <\kappa}$-sequence in $M[G*H*I*J*K]$. Choose $F \in \mathcal{D}_\lambda$. $F$ is a thus a thread through $\overrightarrow{\mathcal{D}}$.

\begin{claim}
 $F \in V[G*H*I*J]$.
\end{claim}

$F \in V[G*H*I*J*K]$. However, $K$ is generic for $j(\mathbb{Q})$, which is $j(\lambda)$-strategically closed in $V[G*H*I*J]$ and thus does not add any $\kappa$-sequences of ordinals. Thus, $F \in V[G*H*I*J]$.

\begin{claim}
 $F \in V[G*H*I]$
\end{claim}

\begin{proof}
Suppose not. Work in $V[G*H*I]$. Then there is an $\mathbb{R}$-name $\dot{F}$ such that $F=\dot{F}^J$ and $\Vdash_\mathbb{R} ``\dot{F}$ is not in the ground model".

Suppose first that $\kappa$ is not strongly inaccessible. Let $\gamma$ be the least cardinal such that $2^\gamma \geq \kappa$. We will construct $\langle p_s \mid s\in {^{\leq \gamma}2} \rangle$ and $\langle \alpha_\beta \mid \beta \leq \gamma \rangle$ satisfying:
\begin{enumerate}
\item{$p_{\langle \rangle} \Vdash ``\dot{F}$ is a thread through $\overrightarrow{\mathcal{D}}$".}
\item{For all $s,u \in {^{\leq \gamma}2}$ such that $s\subseteq u$, we have $p_s, p_u \in \mathbb{R}$ and $p_u\leq p_s$.}
\item{$\langle \alpha_\beta \mid \beta \leq \gamma \rangle$ is a strictly increasing, continuous sequence of ordinals less than $\kappa^+$.}
\item{For all $s\in {^{<\gamma}2}$, there is $\alpha < \alpha_{|s|+1}$ such that $p_{s^\frown \langle 0 \rangle}$ and $p_{s^\frown \langle 1 \rangle}$ decide the statement $``\check{\alpha} \in \dot{F}"$ in opposite ways.}
\item{For all $\beta < \gamma$ and all $s\in {^\beta 2}$, both $p_{s^\frown \langle 0 \rangle}$ and $p_{s^\frown \langle 1 \rangle}$ force that $\dot{F} \cap (\alpha_\beta, \alpha_{\beta+1}) \not= \emptyset$.}
\item{For all limit ordinals $\beta \leq \gamma$ and all $s\in {^\beta 2}$, $p_s \Vdash ``\alpha_\beta$ is a limit point of $\dot{F}$" and there is $D_s \in \mathcal{D}_{\alpha_\beta}$ such that $p_s \Vdash ``\dot{F} \cap \alpha_\beta = D_s"$.}
\end{enumerate}

Suppose for a moment that we have successfully constructed these sequences. Then, for all $s\in {^\gamma 2}$, there is $D_s \in \mathcal{D}_{\alpha_\gamma}$ such that $p_s \Vdash ``\alpha_\gamma$ is a limit point of $\dot{F}$ and $\dot{F} \cap \alpha_\beta = D_s"$. But if $s,u \in {^\gamma 2}$ and $s \not= u$, then there is $\alpha <\alpha_\gamma$ such that $p_s$ and $p_u$ decide the statement $``\alpha \in \dot{F}"$ in opposite ways, so $D_s \not= D_u$. But $2^\gamma \geq \kappa$, so this contradicts the fact that $|\mathcal{D}_{\alpha_\gamma}|<\kappa$.

Now we turn our attention to the construction of such sequences. Fix $p_{\langle \rangle}$ such that $p_{\langle \rangle} \Vdash ``\dot{F}$ is a thread through $\overrightarrow{\mathcal{D}}$", and let $\alpha_0=0$. Fix $\beta <\gamma$ and suppose that $\langle p_s \mid s\in {^\beta 2} \rangle$ and $\alpha_\beta$ have been defined. Fix $s \in {^\beta 2}$. Since $\Vdash_\mathbb{R} ``\dot{F}$ is unbounded in $\kappa^+$", we can find $\alpha > \alpha_\beta$ and $p'_s \leq p_s$ such that $p'_s \Vdash ``\check{\alpha} \in \dot{F}$". Since $\Vdash_\mathbb{R} ``\dot{F}$ is not in the ground model", we can find $\alpha_s > \alpha$ and $p_0, p_1 \leq p'_s$ such that $p_0$ and $p_1$ decide the statement $``\alpha_s \in \dot{F}"$ in opposite ways. Let $p_{s^\frown \langle 0 \rangle} = p_0$ and $p_{s^\frown \langle 1 \rangle} = p_1$. Do this for all $s \in {^\beta 2}$, and let $\alpha_{\beta+1} = \sup(\{ \alpha_s + 1 \mid s\in {^\beta 2} \})$. $2^\beta < \kappa$, so $\alpha_{\beta+ 1} <\kappa^+$.

If $\beta \leq \gamma$ is a limit ordinal and $\langle p_s \mid s\in {^{<\beta} 2} \rangle$ and $\langle \alpha_\delta \mid \delta <\beta \rangle$ have been constructed, let $\alpha_\beta = \sup(\{\alpha_\delta \mid \delta <\beta \})$. Fix $s\in {^\beta 2}$. Since $\mathbb{R}$ is $\kappa$-closed, there is $p\in \mathbb{R}$ such that, for every $\delta < \beta$, $p\leq p_{s\restriction \delta}$. We have arranged that for every $\delta <\beta$ there is $\alpha > \alpha_\delta$ such that $p_{s\restriction (\delta+1)} \Vdash ``\check{\alpha} \in \dot{F}"$. Thus, $p \Vdash ``\check{\alpha}_\beta$ is a limit point of $\dot{F}"$. Find $p' \leq p$ and $D_s \in \mathcal{D}_{\alpha_\beta}$ such that $p' \Vdash ``\dot{F} \cap \check{\alpha}_\beta = \check{D}_s"$. Let $p_s = p'$. Requirements 1-6 above are easily seen to be satisfied by this construction.

Now suppose that $\kappa$ is strongly inaccessible. We modify the above construction slightly. By Fodor's Lemma, find $\nu < \kappa$ and a stationary $S\subseteq S^\lambda_{<\kappa}$ such that if $\alpha \in S$, then $|\mathcal{D}_\alpha| \leq \nu$. Construct $\langle p_s \mid s\in {^{\leq \nu}2} \rangle$ and $\langle \alpha_\beta \mid \beta \leq \nu \rangle$ exactly as above. Fix a sufficiently large regular cardinal $\theta$ and let $M \prec H(\theta)$ contain all relevant information (including $\dot{F}$, $\overrightarrow{\mathcal{D}}$, $\mathbb{R}$, $\langle p_s \mid s\in {^{\leq \nu}2} \rangle$, and $\langle \alpha_\beta \mid \beta \leq \nu \rangle$) such that $|M| = \kappa \subseteq M$ and $\lambda_M=M\cap \lambda \in S$. Fix $\langle \lambda_\eta \mid \eta < \gamma < \kappa \rangle$ increasing and cofinal in $\lambda_M$. Using the $\kappa$-closure of $\mathbb{R}$ and the fact that $\dot{F}$ is forced to be a club, find, for each $s \in {^\nu 2}$, a decreasing sequence of conditions from $\mathbb{R}$, $\langle p_{s, \eta} \mid \eta < \gamma \rangle$ such that, for every $\eta < \gamma$, $p_{s,\eta} \in M$ and there is a $\xi_\eta$ such that $\lambda_\eta < \xi_\eta < \lambda_M$ and $p_{s, \eta} \Vdash ``\check{\xi}_\eta \in \dot{F}"$. Let $p^*_s$ be a lower bound for $\langle p_{s, \eta} \mid \eta < \gamma \rangle$. For each $s \in {^\nu 2}$, $p^*_s \Vdash ``\check{\lambda}_M \in \dot{F}'"$ and, for $s \not= u \in {^\nu 2}$, $p^*_s$ and $p^*_u$ force contradictory information about $\dot{F} \cap \lambda_M$. Since $2^\nu > \nu$, this contradicts the fact that $|\mathcal{D}_{\lambda_M}| \leq \nu$.
\end{proof}

Thus, $F \in V[G*H*I]$. However, $\lambda = (\kappa^+)^{V[G]}$ and, in $V[G]$, $\mathbb{Q}*\dot{\mathbb{T}}$ has a dense $\kappa^+$-closed subset. Thus, $\lambda = (\kappa^+)^{V[G*H*I]}$, contradicting the fact that $F$ is a club in $\lambda$ of order type $\kappa$. 
\end{proof}

Next, we prove that, if $\mu<\nu \leq \kappa$, $\square^\mu(\kappa^+)$ does not imply that there is a stationary $T\subseteq S^{\kappa^+}_\nu$ such that $\square(\kappa^+, T)$ holds. In particular, $\square^\mu(\kappa^+)$ does not imply $\square^\nu(\kappa^+)$. The main idea in the argument, which comes from a modification of the proof of Theorem 18 in \cite{cfm}, is that, though the forcing to thread a $\square(\kappa^+)$-sequence does not necessarily preserve stationary subsets of $S^{\kappa^+}_\nu$, the stationarity of the sets not preserved by the threading forcing is, in a sense, easy to destroy. Thus, by shooting clubs disjoint to these sets, we can arrange so that the threading forcing does in fact preserve stationary subsets of $S^{\kappa^+}_\nu$.

\begin{theorem}
\label{th315}
Let $\mu < \nu \leq \kappa$ be regular cardinals, and let $\lambda > \kappa$ be measurable with $2^\lambda = \lambda^+$. Then there is a forcing extension preserving all cardinals $\leq \kappa$ in which $\square^\mu(\kappa^+)$ holds but in which, for every stationary $T\subseteq S^{\kappa^+}_\nu$, $\square(\kappa^+, T)$ fails.
\end{theorem}

\begin{proof}
Let the initial model be called $V_0$. In $V_0$, let $\mathbb{P} = \mathrm{Coll}(\kappa, <\lambda)$. Let $V=V^\mathbb{P}_0$. Work in $V$. Let $\mathbb{Q} = \mathbb{Q}(\kappa^+)$, and let $\dot{\overrightarrow{C}}$ be a name for the $\square(\kappa^+)$-sequence added by $\mathbb{Q}$. In $V^{\mathbb{Q}}$, let $\mathbb{T} = \mathbb{T}(\dot{\overrightarrow{C}})$.

In $V^{\mathbb{Q}}$, we define a sequence of posets $\langle \mathbb{S}_\alpha \mid \alpha \leq \lambda^+ \rangle$ by induction on $\alpha$. We will show that each $\mathbb{S}_\alpha$ is $\lambda$-distributive and thus does not change any cofinalities $\leq \lambda$. For each $\beta < \lambda^+$, we will fix a $\mathbb{Q}*\mathbb{S}_\beta$-name $\dot{X}_\beta$ for a subset of $S^\lambda_\nu$ such that $\Vdash_{\mathbb{Q}*\mathbb{S}_\beta *\mathbb{T}} ``\dot{X}_\beta$ is non-stationary" and a $\mathbb{Q}*\mathbb{S}_\beta *\mathbb{T}$-name $\dot{E}_\beta$ for a club in $\lambda$ such that $\Vdash_{\mathbb{Q}*\mathbb{S}_\beta *\mathbb{T}} ``\dot{X}_\beta \cap \dot{E}_\beta = \emptyset"$. Elements of $\mathbb{S}_\alpha$ are then functions $s$ such that:

\begin{enumerate}
 \item{$\mathrm{dom}(s) \subseteq \alpha$.}
 \item{$|s| \leq \kappa$.}
 \item{For every $\beta \in \mathrm{dom}(s)$, $s(\beta)$ is a closed, bounded subset of $\lambda$.}
 \item{For every $\beta \in \mathrm{dom}(s)$, $s\restriction \beta \Vdash ``s(\beta) \cap \dot{X}_\beta = \emptyset"$.}
\end{enumerate}
For $s,t \in \mathbb{S}_\alpha$, $t\leq s$ if and only if $\mathrm{dom}(s) \subseteq \mathrm{dom}(t)$ and, for every $\beta \in \mathrm{dom}(s)$, $t(\beta)$ end-extends $s(\beta)$. $\mathbb{S}_{\lambda^+}$ can be seen as a dense subset of an iteration with $\leq \kappa$-support in which the $\alpha^{\mathrm{th}}$ iterand shoots a club disjoint to the interpretation of $\dot{X}_\alpha$. Thus, for each $\alpha < \lambda^+$, $\Vdash_{\mathbb{Q}*\mathbb{S}_{\lambda^+}}``\dot{X}_\alpha$ is non-stationary". By a standard $\Delta$-system argument, it is easy to see that, in $V^\mathbb{Q}$, $\mathbb{S}_{\lambda^+}$ has the $\lambda^+$-chain condition. Thus, since, in $V^\mathbb{Q}$, $2^\lambda = \lambda^+$, we can choose the sequence $\langle \dot{X}_\alpha \mid \alpha < \lambda^+ \rangle$ in such a way that, for every $\beta < \lambda^+$ and every $\mathbb{Q}*\mathbb{S}_\beta$-name $\dot{X}$ for a subset of $S^\lambda_\nu$, if there is $\alpha \geq \beta$ such that $\Vdash_{\mathbb{Q}*\mathbb{S}_\alpha *\mathbb{T}} ``\dot{X}$ is non-stationary", then there is $\alpha^* \geq \alpha$ such that $\Vdash_{\mathbb{Q}*\mathbb{S}_{\alpha^*}}``\dot{X}_{\alpha^*} = \dot{X}"$. Also, again by the $\lambda^+$-chain condition of $\mathbb{S}_{\lambda^+}$, if $\dot{X}$ is a $\mathbb{Q}*\mathbb{S}_{\lambda^+}$-name for a subset of $S^\lambda_\nu$ and $\Vdash_{\mathbb{Q}*\mathbb{S}_{\lambda^+}*\mathbb{T}}``\dot{X}$ is non-stationary", then there is $\alpha < \lambda^+$ and a $\mathbb{Q}*\mathbb{S}_\alpha$-name $\dot{Y}$ for a subset of $S^\lambda_\nu$ such that $\Vdash_{\mathbb{Q}*\mathbb{S}_{\lambda^+}}``\dot{X} = \dot{Y}"$ and $\Vdash_{\mathbb{Q}*\mathbb{S}_\alpha*\mathbb{T}}``\dot{Y}$ is non-stationary". Putting this together, we have that for every $\mathbb{Q}*\mathbb{S}_{\lambda^+}$-name $\dot{X}$ for a subset of $S^{\lambda}_\nu$, if $\Vdash_{\mathbb{Q}*\mathbb{S}_{\lambda^+}*\mathbb{T}}``\dot{X}$ is non-stationary", then already $\Vdash_{\mathbb{Q}*\mathbb{S}_{\lambda^+}}``\dot{X}$ is non-stationary".

\begin{lemma}
In $V^\mathbb{Q}$, for every $\alpha \leq \lambda^+$, $\mathbb{S}_{\alpha}$ is $\nu$-closed .
\end{lemma}

\begin{proof}
 Work in $V^\mathbb{Q}$. The proof is by induction on $\alpha \leq \lambda^+$. Thus, assume that, for all $\beta < \alpha$, $\mathbb{S}_\beta$ is $\nu$-closed. Let $\langle s_\gamma \mid \gamma < \xi \rangle$ be a decreasing sequence from $\mathbb{S}_{\alpha}$, with $\xi < \nu$. We will define a lower bound $s \in \mathbb{S}_\alpha$ for the sequence. Let $\mathrm{dom}(s) = \bigcup_{\gamma < \xi} \mathrm{dom}(s_\gamma)$. Clearly, $|\mathrm{dom}(s)| \leq \kappa$. For $\beta \in \mathrm{dom}(s)$, let $\delta_\beta = \sup(\bigcup_{\gamma < \xi} s_\gamma(\beta))$, and let $s(\beta) = \{\delta_\beta \} \cup \bigcup_{\gamma < \xi} s_\gamma(\beta)$. It is immediate that, if $s \in \mathbb{S}_{\alpha}$, then it is a lower bound for $\langle s_\gamma \mid \gamma < \xi \rangle$. Thus, it remains to show that for all $\beta \in \mathrm{dom}(s)$, $s\restriction \beta \Vdash ``s(\beta) \cap \dot{X}_\beta = \emptyset"$. But, if $\beta \in \mathrm{dom}(s)$, then, by the induction hypothesis, $\mathbb{S}_\beta$ is $\nu$-closed. Therefore $\nu$ remains a regular cardinal in $V^{\mathbb{Q}*\mathbb{S}_\beta}$, so $s\restriction \beta \Vdash_{\mathbb{S}_\beta} ``\check{\delta}_\beta \not\in \dot{X}_\beta"$ and hence $s\restriction \beta \Vdash ``s(\beta) \cap \dot{X}_\beta = \emptyset"$. 
\end{proof}

By genericity, $\overrightarrow{C}$, the square sequence added by $\mathbb{Q}$, is a $\square^\mu(\lambda)$-sequence as witnessed by a stationary $S\subseteq S^\lambda_\mu$. Also, by Lemma \ref{aplem}, $S$ remains stationary in $V^{\mathbb{Q}*\mathbb{S}_{\lambda^+}}$, so $\overrightarrow{C}$ remains a $\square^\mu(\lambda)$-sequence in $V^{\mathbb{Q}*\mathbb{S}_{\lambda^+}}$.

\begin{lemma}
In $V$, for every $\alpha \leq \lambda^+$, $\mathbb{Q}*\mathbb{S}_\alpha * \mathbb{T}$ has a dense $\lambda$-closed subset.
\end{lemma}

\begin{proof}
 Work in $V$. For $\alpha \leq \lambda^+$, let $\mathbb{U}_\alpha$ consist of all conditions of $\mathbb{Q}*\mathbb{S}_\alpha * \mathbb{T}$ of the form $(q, \dot{s}, \dot{t})$ such that $q$ decides the values of $\dot{s}$ and $\dot{t}$, $q \Vdash ``\beta^q = \gamma^t"$, and, for every $\beta \in \mathrm{dom}(\dot{s})$, $(q, \dot{s}\restriction \beta, \dot{t}) \Vdash ``\max(\dot{s}(\beta)) \in \dot{E}_\beta"$. We show by induction on $\alpha$ that $\mathbb{U}_\alpha$ is the desired dense $\lambda$-closed subset.
 
 If $\alpha = 0$, this is simply Proposition \ref{thread}. Let $\alpha = \beta +1$. We first show that $\mathbb{U}_\alpha$ is dense. Fix $(q_0, \dot{s}_0, \dot{t}_0)$. Since $\mathbb{Q}$ is $(\kappa + 1)$-strategically closed, we may assume without loss of generality that $q_0$ decides the value of $\dot{s}_0$ to be $s_0$. Apply the induction hypothesis to get $(q_1, \dot{s}^*_1, \dot{t}_1) \leq (q_0, \check{s}_0 \restriction \beta, \dot{t}_0)$ in $\mathbb{U}_\beta$ such that there is $\gamma < \lambda$ such that $\gamma > \max(s_0(\beta))$ (we say $\max(s_0(\beta))=0$ if $\beta \not\in \mathrm{dom}(s_0)$) and $(q_1, \dot{s}^*_1, \dot{t}_1) \Vdash ``\check{\gamma} \in \dot{E}_\beta"$ (it must then be the case that $(q_1, \dot{s}^*_1) \Vdash ``\check{\gamma} \not\in \dot{X}_\beta"$). Let $\dot{s}_1$ be such that $q_1 \Vdash ``\dot{s}_1 = \dot{s}^*_1 \cup \{(\beta, s_0(\beta) \cup \{\gamma \}) \}"$. Then $(q_1, \dot{s}_1, \dot{t}_1)$ extends $(q_0, \dot{s}_0, \dot{t}_0)$ and is in $\mathbb{U}_\alpha$.
 
 We now show that $\mathbb{U}_\alpha$ is $\lambda$-closed. Let $\xi < \lambda$, and let $\langle (q_\eta, \dot{s}_\eta, \dot{t}_\eta) \mid \eta < \xi \rangle$ be a decreasing sequence from $\mathbb{U}_\alpha$. By the methods of the proof of Proposition \ref{thread}, we can find $(q, \dot{t})$ such that $q$ decides the value of $\dot{t}$, $q \Vdash ``\beta^q = \gamma^t"$, and, for every $\eta<\xi$, $(q,\dot{t}) \leq (q_\eta, \dot{t}_\eta)$. We now define an $s$ so that $(q, \check{s}, \dot{t})$ is a lower bound for our sequence. Let $\mathrm{dom}(s) = \{\delta \mid \mbox{for some } \eta < \xi, q_\eta \Vdash \delta \in \mathrm{dom}(\dot{s}_\eta) \}$. For $\delta \in \mathrm{dom}(s)$, let $r_\delta = \{\gamma \mid \mbox{for some } \eta < \xi, q_\alpha \Vdash \gamma \in \dot{s}_\eta(\delta) \}$ and let $s(\delta) = r_\delta \cup \sup(r_\delta)$. All that remains to be shown is that for every $\delta \in \mathrm{dom}(s)$, $(q, \check{s} \restriction \delta, \dot{t}) \Vdash ``\max(\check{s}(\delta)) \in \dot{E}_\delta"$. We show this by induction on $\delta$. Thus, assume it holds for all ordinals less than $\delta$ in $\mathrm{dom}(s)$. Then $(q, \check{s} \restriction \delta, \dot{t}) \in \mathbb{U}_\delta$. For every $\eta < \xi$ with $\delta \in \mathrm{dom}(s_\eta)$, $(q, \dot{s}_\eta \restriction \delta, \dot{t}) \Vdash ``\max(\dot{s}_\eta (\delta)) \in \dot{E}_\delta"$. Thus, $(q, \check{s} \restriction \delta, \dot{t}) \Vdash ``\dot{E}_\delta \mbox{ is unbounded in } \check{r}_\delta "$. Since $\dot{E}_\delta$ is forced to be a club, $(q, \check{s} \restriction \delta, \dot{t}) \Vdash ``\sup(r_\delta) = \max(\check{s}(\delta)) \in \dot{E}_\delta"$.
 
 Now suppose that $\alpha \leq \lambda^+$ is a limit ordinal. To show that $\mathbb{U}_\alpha$ is dense, let $(q_0, \dot{s}_0, \dot{t}_0) \in \mathbb{Q}*\mathbb{S}_\alpha*\mathbb{T}$. Assume without loss of generality that $q_0$ decides the value of $\dot{s}_0$ and $\dot{t}_0$. If $\mathrm{cf}(\alpha) = \lambda$, then there is $\beta < \alpha$ such that $q_0 \Vdash ``\mathrm{dom}(\dot{s}_0) \subseteq \beta"$. Then $(q_0, \dot{s}_0, \dot{t}_0) \in \mathbb{Q}*\mathbb{S}_\beta * \mathbb{T}$ and, by the inductive hypothesis, we can find $(q_1, \dot{s}_1, \dot{t}_1) \leq (q_0, \dot{s}_0, \dot{t}_0)$ such that $(q_1, \dot{s}_1, \dot{t}_1) \in \mathbb{U}_\beta \subseteq \mathbb{U}_\alpha$.
 
 If $\mathrm{cf}(\alpha) < \lambda$, fix an increasing, continuous sequence $\langle \alpha_i \mid i< \xi \rangle$ cofinal in $\alpha$, with $\xi \leq \kappa$. We define a sequence $\langle (q_i, \dot{s}_i, \dot{t}_i) \mid 1 \leq i<\xi \rangle$ such that, for all $1\leq i<j<\xi$,
 
 \begin{itemize}
 \item{$(q_i, \dot{s}_i, \dot{t}_i) \in \mathbb{U}_{\alpha_i}$.}
 \item{$(q_i, \dot{s}_i, \dot{t}_i) \leq (q_0, \dot{s}_0 \restriction \alpha_i, \dot{t}_0)$.}
 \item{$(q_j, \dot{s}_j, \dot{t}_j) \leq (q_i, \dot{s}_i, \dot{t}_i)$.}
 \end{itemize}
 
 If $i = i'+1$, define $(q_i, \dot{s}_i, \dot{t}_i)$ as follows. Let $\dot{s}^*_i$ be such that $q_{i'} \Vdash ``\dot{s}^*_i = \dot{s}_{i'} \cup \dot{s}_0 \restriction [\alpha_{i'}, \alpha_i)"$. By the inductive hypothesis, find $(q_i, \dot{s}_i, \dot{t}_i) \in \mathbb{U}_{\alpha_i}$ such that $(q_i, \dot{s}_i, \dot{t}_i) \leq (q_{i'}, \dot{s}^*_i, \dot{t}_{i'})$. If $i$ is a limit ordinal, then, by the inductive hypothesis, $\mathbb{U}_{\alpha_i}$ is $\lambda$-closed, so we can find $(q_i, \dot{s}_i, \dot{t}_i) \in \mathbb{U}_{\alpha_i}$ that is a lower bound for $\langle (q_j, \dot{s}_j, \dot{t}_j) \mid j<i \rangle$.
 
 We now define $(q, \dot{s}, \dot{t}) \in \mathbb{U}_\alpha$ which is a lower bound for $\langle (q_i, \dot{s}_i, \dot{t}_i) \mid i<\xi \rangle$. First, by previous arguments, find $(q, \dot{t}) \in \mathbb{U}_0$ which is a lower bound for $\langle (q_i, \dot{t}_i) \mid i<\xi \rangle$. Next, let $X = \{\beta \mid \mbox{for some } i<\xi, q_i \Vdash ``\check{\beta} \in \mathrm{dom}(\dot{s}_i) \}$. If $\beta \in X$, let $\dot{s}^*(\beta)$ be such that $q \Vdash ``\dot{s}^*(\beta) = \bigcup_{i<\xi} \dot{s}_i(\beta)$. Let $\dot{s}$ be such that $q \Vdash ``\mathrm{dom}(\dot{s}) = X \mbox{ and, if } \beta \in X, \mbox{ then } \dot{s}(\beta) = \dot{s}^*(\beta) \cup \{\sup(\dot{s}^*(\beta) \}"$. It is routine to check that $(q, \dot{s}, \dot{t}) \in \mathbb{U}_\alpha$ and $(q, \dot{s}, \dot{t}) \leq (q_0, \dot{s}_0, \dot{t}_0)$.
 
 The proof that $\mathbb{U}_\alpha$ is $\lambda$-closed is the same as in the successor case. 
 \end{proof}
 
 It follows that, for all $\alpha \leq \lambda^+$, $\mathbb{S}_\alpha$ is $\lambda$-distributive and thus preserves all cardinals and cofinalities. It remains to show that, in $V^{\mathbb{Q}*\mathbb{S}_{\lambda^+}}$, $\square(\lambda, T)$ fails for every stationary $T \subseteq S^{\lambda}_\nu$.
 
 Fix $j: V_0 \rightarrow M$ with $\mathrm{crit}(j) = \lambda$. Let $G$ be $\mathbb{P}$-generic over $V_0$, let $H$ be $\mathbb{Q}$-generic over $V_0[G] = V$, let $I$ be $\mathbb{S}_{\lambda^+}$-generic over $V[H]$, and let $J$ be $\mathbb{T}$-generic over $V[H*I]$. Since, in $V$, $\mathbb{Q}*\mathbb{S}_{\lambda^+}*\mathbb{T}$ has a $\lambda$-closed dense subset and has size $\lambda^+$, by Fact \ref{lift} we can extend the identity map $i: \mathbb{P} \rightarrow j(\mathbb{P})$ to a complete embedding $i^*: \mathbb{P}*\mathbb{Q}*\mathbb{S}_{\lambda^+}*\mathbb{T} \rightarrow j(\mathbb{P})$ such that the quotient forcing $\mathbb{R} = j(\mathbb{P})/i^*[\mathbb{P}*\mathbb{Q}*\mathbb{S}_{\lambda^+}*\mathbb{T}]$ is $\kappa$-closed. Thus, letting $K$ be $\mathbb{R}$-generic over $V[H*I*J]$, we can lift $j$ to an elementary embedding $j: V \rightarrow M[G*H*I*J*K]$.
 
 Suppose now for sake of contradiction that, in $V[H*I]$, $T\subseteq S^\lambda_\nu$ is stationary and $\overrightarrow{D} = \langle D_\alpha \mid \alpha < \lambda \rangle$ is a $\square(\lambda, T)$-sequence. For $\xi < \lambda^+$, let $I_\xi = I \cap \mathbb{S}_\xi$. Each $I_\xi$ is then $\mathbb{S}_\xi$-generic over $V[H]$. Since $\mathbb{S}_{\lambda^+}$ has the $\lambda^+$-c.c., we can fix $\xi^* < \lambda^+$ such that $T, \overrightarrow{D} \in V[H*I_{\xi^*}]$.
 
 We would like to lift $j$ further to have domain $V[H*I_{\xi^*}]$. To do this, we define a master condition $(q^*, \check{s}^*, \check{t}^*) \in j(\mathbb{Q}*\mathbb{S}_{\xi^*}*\mathbb{T})$. $q^*$ is defined exactly as in the proof of Theorem \ref{th312}. Let $E$ be the club added by $J$, and let $t^* = E \cup \{\lambda\}$. Then $(q^*, \check{t}^*) \in j(\mathbb{Q}*\mathbb{T})$. Let $s$ be the generic object added by $I_{\xi^*}$. $s$ is thus a function with domain $\xi^*$, where, for each $\alpha < \xi^*$, $s(\alpha)$ is a club in $\lambda$. Let $s^*$ be such that $\mathrm{dom}(s^*) = j[\xi^*]$ and, for each $\alpha < \xi^*$, $s^*(j(\alpha)) = s(\alpha) \cup \{\lambda\}$. It is clear that, if $(q^*, \check{s}^*) \in j(\mathbb{Q}*\mathbb{S}_{\xi^*})$, then it is a lower bound for $j[H*I_{\xi^*}]$. Thus, all that needs to be checked is that, for every $\alpha < \xi^*$, $(q^*, \check{s}^* \restriction j(\alpha)) \Vdash ``\check{s}^*(j(\alpha)) \cap j(\dot{X}_\alpha) = \emptyset"$. We show this by induction on $\alpha$. It suffices to show that $(q^*, \check{s}^* \restriction j(\alpha)) \Vdash ``\check{\lambda} \not\in j(\dot{X}_\alpha)"$. Suppose for sake of contradiction that there is $(q', \dot{s}') \leq (q^*, \check{s}^* \restriction j(\alpha))$ such that $(q',\dot{s}') \Vdash ``\check{\lambda} \in j(\dot{X}_\alpha)"$. Recall that $\dot{E}_\alpha$ is a $\mathbb{Q}*\mathbb{S}_\alpha*\mathbb{T}$-name for a club in $\lambda$. $(q', \dot{s}', \check{t}^*)$ is a lower bound for $H*I\restriction \alpha *J$, so, for every $\beta <\lambda$, $(q', \dot{s}', \check{t}^*) \Vdash ``\check{\beta} \in j(\dot{E}_\alpha)"$ if and only if $\beta \in E_\alpha$. Thus, $(q', \dot{s}', \check{t}^*) \Vdash ``\lambda \mbox{ is a limit point of } j(\dot{E}_\alpha)"$, so $(q', \dot{s}', \check{t}^*) \Vdash ``\lambda \not\in j(\dot{X}_\alpha)"$. This is a contradiction.
 
 Thus, $(q^*, \check{s}^*)$ is a lower bound for $j[H*I_{\xi^*}]$ in $j(\mathbb{Q}*\mathbb{S}_{\xi^*})$, so, if we let $H^+*I_{\xi^*}^+$ be $j(\mathbb{Q}*\mathbb{S}_{\xi^*})$-generic over $M[G*H*I*J*K]$ with $(q^*, \check{s}^*) \in H^+*I_{\xi^*}^+$, then we can lift $j$ to an elementary embedding $j:V[H*I_{\xi^*}] \rightarrow M[G*H*I*J*K*H^+*I_{\xi^*}^+]$. 
 
 Now $\overrightarrow{D}$ and $T$ are in the domain of $j$, $j(\overrightarrow{D}) = \langle D_\alpha \mid \alpha < j(\lambda) \rangle$ is a $\square(j(\lambda), j(T))$-sequence, and $j(T) \cap \lambda = T$. Thus, $D_\lambda$ is a thread through $\overrightarrow{D}$ and avoids $T$. Since $j(\mathbb{Q}*\mathbb{S}_{\xi^*})$ is $j(\lambda)$-distributive, $H^+*I_{\xi^*}^+$ could not have added $D_\lambda$, so $D_\lambda \in V[H*I*J*K]$. Also, since $K$ is generic for $\kappa$-closed forcing, the argument from Theorem \ref{th312} shows that it can not add a thread through $\overrightarrow{D}$, so $D_\lambda \in V[H*I*J]$ and witnesses that $T$ is not stationary in $V[H*I*J]$. But we arranged our iteration $S_{\lambda^+}$ so that this implies that $T$ is already non-stationary in $V[H*I]$. This is a contradiction, so $\square(\lambda, T)$ fails in $V[H*I]$. 
\end{proof}

Finally, we show that the existence of a stationary $S\subset S^{\kappa^+}_\kappa$ such that $\square(\kappa^+, S)$ holds does not imply the existence of a stationary $T\subset S^{\kappa^+}_{<\kappa}$ for which $\square(\kappa^+, T)$ holds. In \cite{hs}, Harrington and Shelah show that, after collapsing a Mahlo cardinal to be $\kappa^+$, one can iteratively force to shoot clubs disjoint to non-reflecting subsets of $S^{\kappa^+}_{<\kappa}$, thus obtaining a model in which every stationary subset of $S^{\kappa^+}_{<\kappa}$ reflects. We will work in L and carry out the forcing iteration used by Harrington and Shelah, arguing that our desired conclusion holds in the final model. We use the following theorem of Jensen about squares in L (\cite{devlin}, Chapter VII). 

\begin{theorem}
\label{jens}
Suppose $V=L$. Let $\lambda$ be an inaccessible cardinal which is not weakly compact, and let $S\subseteq \lambda$ be stationary. Then there is a stationary $S'\subseteq S$ such that $\square(\lambda, S')$ holds. 
\end{theorem}

\begin{theorem}
 Suppose $V=L$, $\kappa$ is a regular, uncountable cardinal, and $\lambda > \kappa$ is the least Mahlo cardinal greater than $\kappa$. Then there is a forcing extension in which $\lambda = \kappa^+$, there is a stationary $S\subseteq S^\lambda_\kappa$ such that $\square(\lambda, S)$ holds, and all stationary subsets of $S^\lambda_{<\kappa}$ reflect (and hence $\square(\lambda, T)$ fails for every stationary $T\subseteq S^\lambda_{<\kappa}$).
 \end{theorem}

\begin{proof}
By Theorem \ref{jens}, fix a stationary $S\subseteq \lambda$ consisting of inaccessible cardinals and a $\square(\lambda, S)$-sequence, $\overrightarrow{C}$. Let $\mathbb{P} = \mathrm{Coll}(\kappa, <\lambda)$. In $V^\mathbb{P}$, we define an iteration $\langle \mathbb{Q}_\alpha \mid \alpha \leq \lambda^+ \rangle$, which will shoot clubs disjoint to non-reflecting sets of ordinals, by induction on $\alpha$. For each $\alpha < \lambda^+$, we will fix a $\mathbb{Q}_\alpha$-name $\dot{X}_\alpha$ such that $\Vdash_{\mathbb{Q}_\alpha} ``\dot{X}_\alpha \subseteq S^\lambda_{<\kappa}$ and $\dot{X}_\alpha$ does not reflect at any ordinal of uncountable cofinality". Elements of $\mathbb{Q}_\alpha$ are functions $q$ such that:
 \begin{enumerate}
  \item{$\mathrm{dom}(q)\subseteq \alpha$.}
  \item{$|q| \leq \kappa$.}
  \item{For every $\beta \in \mathrm{dom}(q)$, $q(\beta)$ is a closed, bounded subset of $\lambda$.}
  \item{For every $\beta \in \mathrm{dom}(q)$, $q\restriction \beta \Vdash ``q(\beta) \cap \dot{X}_\beta = \emptyset"$.}
 \end{enumerate}
 For $p,q \in \mathbb{Q}_\alpha$, $q\leq p$ if and only if $\mathrm{dom}(p)\subseteq \mathrm{dom}(q)$ and, for every $\beta \in \mathrm{dom}(p)$, $q(\beta)$ end-extends $p(\beta)$. An easy $\Delta$-system argument shows that $\mathbb{Q}_{\lambda^+}$ has the $\lambda^+$-c.c. Thus, with a suitable choice of the names $\dot{X}_\alpha$, we can arrange that, in $V^{\mathbb{P}*\mathbb{Q}_{\lambda^+}}$, every stationary subset of $S^\lambda_{<\kappa}$ reflects.
 
 Back in $V$, fix a sufficiently large, regular cardinal $\theta$ (in particular, $\theta > \lambda^+$). Let $\mathcal{N}$ be the set of $N$ such that $\mathbb{P}\in N$, $N \preceq (H(\theta), \in, \lhd)$ (where $\lhd$ is a well-ordering of $H(\theta)$),  $\lambda_N := N\cap \lambda$ is an inaccessible cardinal, $|N| = \lambda_N$, and $N^{<\lambda_N} \subseteq N$. For $N \in \mathcal{N}$, let $\pi_N: N \rightarrow \bar{N}$ be the transitive collapse. If $x\in N$, let $x_N = \pi_N(x)$.
 
 \begin{lemma}
 \label{lem320}
  For every $x\in H(\theta)$, there is $N\in \mathcal{N}$ such that $x\in N$.
 \end{lemma}
 
 \begin{proof}
  Fix $x\in H(\theta)$. We find the desired $N \in \mathcal{N}$ by building an increasing, continuous chain $\langle N_\alpha \mid \alpha<\lambda \rangle$ such that, for each $\alpha < \lambda$, 
  \begin{enumerate}
   \item{$x, \lambda \in N_\alpha$}
   \item{$N_\alpha \prec H(\theta)$.}
   \item{$|N_\alpha| < \lambda$.}
   \item{$\mathcal{P}(N_\alpha) \subseteq N_{\alpha+1}$.}
   \item{$\sup(N_\alpha \cap \lambda) \subseteq N_{\alpha+1}$.}
  \end{enumerate}
  Let $E$ be the set of $\alpha < \lambda$ such that $N_\alpha \cap \lambda = \alpha = |N_\alpha|$. $E$ is a club in $\lambda$, so, since $\lambda$ is Mahlo, there is $\alpha^* \in E$ such that $\alpha^*$ is inaccessible. $N_{\alpha^*}$ is then in $\mathcal{N}$. 
 \end{proof}
 
 If $N \in \mathcal{N}$, then, since $N$ is closed under $<\lambda_N$-sequences, $\mathbb{P}_N = \mathbb{P} \cap N = \mathrm{Coll}(\kappa, <\lambda_N)$, so $N^\mathbb{P} \prec H(\theta)^\mathbb{P}$. Also, since $\mathbb{P}_N$ has the $\lambda_N$-c.c., $\bar{N}^{\mathbb{P}_N} \cong N^\mathbb{P}$ is closed under $<\lambda_N$-sequences from $V^{\mathbb{P}_N}$.
 
The following Lemma, which will show that $\mathbb{Q}_{\lambda^+}$ is $\lambda$-distributive, is proven in \cite{hs}.
 
 \begin{lemma}
 \label{lem321}
 Let $\alpha < \lambda^+$.
  \begin{enumerate}
   \item{For all $N\in \mathcal{N}$ such that $\alpha \in N$, in $V^{\mathbb{P}_N}$, $(\mathbb{Q}_\alpha)_N$ has a $\lambda_N$-closed dense subset.}
      \item{In $V^\mathbb{P}$, $\mathbb{Q}_\alpha$ is $\lambda$-distributive.}
      \item{For all $N\in \mathcal{N}$ such that $\alpha \in N$, in $V^{\mathbb{P}_N*(\mathbb{Q}_\beta)_N}$, $(X_\beta)_N$ is not stationary in $\lambda_N$.}
  \end{enumerate}
 \end{lemma}
 
 By the chain condition, every $<\lambda$-sequence in $V^{\mathbb{P}*\mathbb{Q}_{\lambda^+}}$ appears in $V^{\mathbb{P}*\mathbb{Q}_\alpha}$ for some $\alpha < \lambda^+$, so we have that, in $V^\mathbb{P}$, $\mathbb{Q}_{\lambda^+}$ is $\lambda$-distributive and thus preserves all cardinals $\leq \lambda$.
 
 \begin{lemma}
  In $V^{\mathbb{P}*\mathbb{Q}_{\lambda^+}}$, there is no stationary $T\subseteq S^\lambda_{<\kappa}$ such that $\square(\lambda, T)$ holds.
 \end{lemma}
 
 \begin{proof}
  Suppose there is such a $T$, and let $\overrightarrow{D}$ be a $\square(\lambda, T)$-sequence. Then, for each $\alpha < \lambda$ of uncountable cofinality, $D'_\alpha$ witnesses that $T\cap \alpha$ is non-stationary, so $T$ does not reflect. However, in $V^{\mathbb{P}*\mathbb{Q}_{\lambda^+}}$, every stationary subset of $S^\lambda_{<\kappa}$ reflects. Contradiction. 
 \end{proof}
 
 In $V^{\mathbb{P}*\mathbb{Q}_{\lambda^+}}$, $\overrightarrow{C}$ is clearly still a coherent sequence avoiding $S$ and $S\subseteq S^\lambda_\kappa$. Thus, the following lemma suffices to prove the theorem.
 
 \begin{lemma}
  $S$ is stationary in $V^{\mathbb{P}*\mathbb{Q}_{\lambda^+}}$.
 \end{lemma}
 
 \begin{proof}
  Let $E \in V^{\mathbb{P}*\mathbb{Q}_{\lambda^+}}$ be a club in $\lambda$. By the chain condition, there is $\alpha < \lambda^+$ such that $E\in V^{\mathbb{P}*\mathbb{Q}_\alpha}$. Thus, it suffices to show that $S$ remains stationary in $V^{\mathbb{P}*\mathbb{Q}_{\alpha}}$ for every $\alpha<\lambda^+$.
  
  To this end, fix $\alpha<\lambda^+$, $(p, \dot{q}) \in \mathbb{P}*\mathbb{Q}_\alpha$, and $\dot{E}$, a $\mathbb{P}*\mathbb{Q}_\alpha$-name for a club in $\lambda$. By the argument from the proof of Lemma \ref{lem320}, we can find $N\in \mathcal{N}$ such that $\{(p, \dot{q}),\dot{E}, \alpha \} \subseteq N$ and $\lambda_N \in S$. Let $G$ be $\mathbb{P}$-generic over V with $p \in G$, and let $G_N$ be the restriction of $G$ to $\mathrm{Coll}(\kappa, <\lambda_N)$. Since $\mathbb{P}$ has the $\lambda$-c.c., $S$ is still stationary in $V[G]$. Let $q$ be the interpretation of $\dot{q}$ in $V[G]$, and reinterpret $\dot{E}$ in $V[G]$ as a $\mathbb{Q}_\alpha$-name. Also, we can extend $\pi_N$ to an isomorphism of $N[G]$ and $\bar{N}[G_N]$. Enumerate the dense open sets of $(\mathbb{Q}_\alpha)_N$ lying in $\bar{N}[G_N]$ as $\langle D_\xi \mid \xi < \lambda_N \rangle$. By Lemma \ref{lem321}(1) and the fact that $\bar{N}[G_N]$ is closed under $<\lambda_N$-sequences, we can find a decreasing sequence $\langle q_\xi \mid \xi < \lambda_N \rangle$ of conditions from the $\lambda_N$-closed dense subset of $(\mathbb{Q}_\alpha)_N$ such that $q_0 = q$ and, for all $\xi < \lambda_N$, $q_{\xi+1} \in D_\xi \cap \bar{N}[G_N]$. We define $q^*$ to be a lower bound for $\langle \pi_N^{-1}(q_\xi) \mid \xi < \lambda_N \rangle$ in $\mathbb{Q}_\alpha$ by letting $\mathrm{dom}(q^*) = N \cap \alpha$ and, for each $\beta \in \mathrm{dom}(q^*)$, \[q^*(\beta) = \bigcup_{\xi < \lambda_N} q_\xi(\pi_N(\beta)) \cup \{\lambda_N \}.\] cf($\lambda_N) = \kappa$ in $V^{\mathbb{P}*\mathbb{Q}_\beta}$, so $\lambda_N$ is forced not to be in $\dot{X}_\beta$ and thus $q^* \in \mathbb{Q}_\alpha$. For every $\gamma < \lambda_N$, there is $\xi < \lambda_N$ and $\delta \in (\gamma, \lambda_N)$ such that $q_\xi \Vdash_{(\mathbb{Q}_\alpha)_N} ``\check{\delta} \in \pi(\dot{E})"$. Thus, $\pi_N^{-1}(q_\xi) \Vdash_{\mathbb{Q}_\alpha} ``\check{\delta} \in \dot{E}"$. so $q^* \Vdash_{\mathbb{Q}_\alpha} ``\dot{E} \mbox{ is unbounded in } \check{\lambda}_N"$. Since $\dot{E}$ is a name for a club, $q^* \Vdash_{\mathbb{Q}_\alpha} ``\lambda_N \in \dot{E} \cap \check{S}"$, so $S$ is stationary in $V^{\mathbb{P}*\mathbb{Q}_\alpha}$. 
 \end{proof} 
 \renewcommand{\qedsymbol}{}
 \end{proof}
 We conclude with a diagram illustrating the situation at $\omega_2$, where we now have a complete picture. Arrows correspond to implications, and struck-out arrows to non-implications.
 
 \[
\begin{tikzpicture}[x=3cm, y=3cm]
\node (sq0) at (1,3) {$\square_{\omega_1}$};
\node (sq1) at (0,2) {$\square^\omega(\omega_2)$};
\node (sq2) at (2,2) {$\square^{\omega_1}(\omega_2)$};
\node (sq3) at (0,1) {$\exists$ stationary $S\subseteq S^{\aleph_2}_{\aleph_0}(\square (\omega_2, S))$};
\node (sq4) at (2,1) {$\exists$ stationary $T\subseteq S^{\aleph_2}_{\aleph_1}(\square (\omega_2, T))$};
\node (sq5) at (1,0) {$\square(\omega_2)$};
\path 
(sq0) edge [->, bend right=10] (sq1)
(sq1) edge [->, bend right=10] node[strike out, draw, -,sloped]{} (sq0)
(sq0) edge [->, bend right=10] (sq2)
(sq2) edge [->, bend right=10] node[strike out, draw,-]{} (sq0)
(sq1) edge [->, bend right=10] node[strike out, draw,-]{} (sq2)
(sq2) edge [->, bend right=10] (sq1)
(sq1) edge [->, bend right=10] (sq3)
(sq3) edge [->, bend right=10] (sq1)
(sq2) edge [->, bend right=10] (sq4)
(sq4) edge [->, bend right=10] node[strike out, draw,-]{} (sq2)
(sq3) edge [->, bend right=10] node[strike out, draw,-]{} (sq4)
(sq4) edge [->, bend right=10] node[strike out, draw,-]{} (sq3)
(sq3) edge [->, bend right=10] (sq5)
(sq5) edge [->, bend right=10] node[strike out, draw,-]{} (sq3)
(sq4) edge [->, bend right=10] (sq5)
(sq5) edge [->, bend right=10] node[strike out, draw,-,sloped]{} (sq4);
\end{tikzpicture}
\]
 
 \bibliography{squareref}
 \bibliographystyle{plain}

\end{document}